\documentclass[a4paper,12pt]{article}
\usepackage{amsfonts}
\usepackage{amsmath}
\usepackage{amssymb}
\usepackage{amsthm}
\usepackage{graphicx}
\usepackage{tikz}
\usepackage[hidelinks]{hyperref}
\usepackage[inline]{enumitem}
\usepackage{subcaption}

\usetikzlibrary{arrows,shapes, positioning, matrix, patterns}

\numberwithin{equation}{section}

\newtheorem{prop}{Proposition}
\numberwithin{prop}{section}
\newtheorem{lem}{Lemma}
\numberwithin{lem}{section}

\numberwithin{coro}{section}
\newtheorem{teo}{Theorem}
\numberwithin{teo}{section}

\theoremstyle{definition}
\newtheorem{defix}{Definition}
\numberwithin{defix}{section}
\newenvironment{defi}
  {\pushQED{\qed}\defix}
  {\popQED\enddefix}
\newtheorem{exex}{Example}
\numberwithin{exex}{section}
\newenvironment{exe}
  {\pushQED{\qed}\exex}
  {\popQED\endexex}

\newtheorem{obsx}{Remark}
\numberwithin{obsx}{section}
\newenvironment{obs}
  {\pushQED{\qed}\obsx}
  {\popQED\endobsx}

\newcommand{\ran}{\textnormal{range}}
\newcommand{\spn}{\textnormal{span}}

\providecommand{\keywords}[1]{\textit{Keywords:  } #1}

\providecommand{\MSC}[1]{\textit{2010 Mathematics subject classification:  } #1}

\title{Synchrony Branching Lemma \\for Regular Networks}
\author{
Pedro Soares\thanks{Departamento de  Matem\'{a}tica,
Faculdade de Ci\^{e}ncias  da Universidade do Porto,
Centro de Matem\'{a}tica da Universidade do Porto
    (ptcsoares@fc.up.pt).\newline This work was funded by FCT (Portugal) through the PhD grant PD/BD/105728/2014 and partially supported by CMUP (UID/MAT/00144/ 2013), which is funded by FCT with national (MEC) and European structural funds (FEDER), under the partnership agreement PT2020.}
}
\date{\vspace{-13mm}}

\begin{document}
\maketitle

\begin{abstract}
Coupled cell systems are dynamical systems associated to a network and synchrony subspaces, given by balanced colorings of the network, are invariant subspaces for every coupled cell systems associated to that network. 
Golubitsky and Lauterbach (SIAM J. Applied Dynamical Systems, 8 (1) 2009, 40–-75) prove an analogue of the Equivariant Branching Lemma in the context of regular networks. 
We generalize this result proving the generic existence of steady-state bifurcation branches for regular networks with maximal synchrony. 
We also give necessary and sufficient conditions for the existence of steady-state bifurcation branches with some submaximal synchrony. Those conditions only depend on the network structure, but the lattice structure of the balanced colorings is not sufficient to decide which synchrony subspaces support a steady-state bifurcation branch.

\vspace{1em}
\hspace{-1.8em}
\keywords{Coupled cell systems; Steady-state bifurcations; Synchrony-breaking bifurcations.}

\vspace{1em}
\hspace{-1.8em}
\MSC{37G10; 34D06; 34C23}
\end{abstract}

\section{Introduction}

Coupled cell networks describe influences between cells and can be represented by graphs. 
A dynamical system that respects a network structure is called a coupled cell system associated to the network. 
In \cite{SGP03} and \cite{GST05}, the authors formalize the concepts of (coupled cell) network and coupled cell system. 
They also show that there exists an intrinsic relation between coupled cell systems and networks, proving that a polydiagonal subspace given by a coloring of the network is an invariant subspace for any coupled cell system if and only if the coloring is balanced. 
Here, a coloring is balanced if any two cells with the same color receive, for each color, the same number of inputs starting in cells with that color. 
And a polydiagonal subspace given by a balanced coloring is called a synchrony subspace.
Given a balanced coloring, they define the corresponding quotient network by merging cells with equal color. 
Moreover, the restriction of a coupled cell system to a synchrony subspace is a coupled cell system for the quotient.
We will focus on regular networks, where all cells and edges are identical, and each regular network can be represented by an adjacency matrix. 
The adjacency matrix of a quotient network is given by the restriction of the original network adjacency matrix to the corresponding synchrony subspace, \cite{ADGL09}.


Equivariant theory is the study of dynamical systems that commute with an action of a group in the phase space and isotropy subgroups are the subgroups that fix some point of the phase space, see e.g. \cite{F07}. 
For each isotropy subgroup, the set of fixed points forms an invariant subspace for every equivariant dynamical system and it is called the fixed point subspace. 
One goal of equivariant bifurcation theory is to characterize which isotropy subgroups support a bifurcation. 
The Equivariant Branching Lemma \cite{C81} is one of the first important results about the existence of symmetry-breaking steady-state bifurcation branches for isotropy subgroups that have one dimensional fixed point subspaces. 
Later this result was extended for isotropy subgroups that have odd dimensional fixed point subspaces, see e.g. \cite{LLH94, CL00}.
This topic is a large source of inspiration to the study of synchrony-breaking bifurcations on networks, where one of the key questions concerns the characterization of synchrony subspaces which support (generically) steady-state bifurcation branches.


A similar result to the first version of the Equivariant Branching Lemma for regular networks has been already stated, see \cite[Theorem 2.1]{WG05}, \cite[Theorem 6.3]{GL09} and \cite[Corollary 3.1.]{K09}, we call this result the Synchrony Branching Lemma.
The eigenvalues of the Jacobian of a coupled cell system associated to a regular network at a full synchronous solution are related to the eigenvalues of its adjacency matrix and this relation preserves multiplicities, \cite{LG06}.
So, we can use the eigenvalue structure of the network adjacency matrix to tabulate the possible local codimension-one (steady-state or Hopf) synchrony-breaking bifurcations that can occur for the coupled cell systems associated to a network.
We say that an eigenvalue of the adjacency matrix belongs to a balanced coloring, if it has an eigenvector in the synchrony subspace given by that coloring.
Fixing an eigenvalue, we say that a coloring is maximal if the eigenvalue belongs to that coloring and it does not belong to any lower dimensional synchrony subspace.
The Synchrony Branching Lemma states that every synchrony subspace given by a maximal coloring with a simple eigenvalue (algebraic multiplicity $1$) generic supports a bifurcation branch. 
In \cite{SG11}, the authors study the degeneracy of steady-state bifurcation problems for regular networks and simple eigenvalues. 
They give conditions on the network structure for the degeneracy of steady-state bifurcation problems and they also present examples of regular networks  that have generic highly degenerated steady-state bifurcation problems. 
In \cite{K09}, it is given a characterization of the synchrony subspaces which support a synchrony-breaking bifurcation using the lattice structure of balanced colorings, for regular networks that only have simple eigenvalues.


In this manuscript, we generalize the Synchrony Branching Lemma for semisimple eigenvalues (the algebraic and geometric multiplicity are equal). 
We prove that a synchrony subspace given by a maximal coloring generically supports a bifurcation branch, if the semisimple eigenvalue has odd multiplicity, \ref{prop:bifbraoddker}.
This follows from the application of the Lyapunov-Schmidt Reduction \cite{GS85} and a blow-up technique also used in equivariant bifurcation, see e.g. \cite{LLH94}. 
In the way, we prove that the degeneracy of a bifurcation problem associated to a semisimple eigenvalue only depends on the network structure, \ref{lem:degnetstr}. 
Next, we focus on semisimple eigenvalues with multiplicity $2$.
If a coloring is maximal and has even degeneracy, then its synchrony subspace supports a bifurcation branch, \ref{prop:bifbranchdim2max}.
We also give necessary and sufficient conditions for the existence of bifurcation branches on synchrony subspaces given by submaximal colorings,  i.e., the eigenvalue belongs to the submaximal coloring and it must have multiplicity $0$ or $1$ in any synchrony subspace strictly included in the synchrony subspace given by the submaximal coloring, \ref{prop:bifbranchdim2submax}.
Those conditions only depend on the network structure.
We give examples of networks where the previous results apply, including two networks that have the same synchrony lattice structure but do not have the same type of synchrony-breaking bifurcations, \ref{exe:latbalcoldegdet} and \ref{exe:latbalcoldegdet2}.
Despite we do not present an explicit network, we show how a network can have a semisimple eigenvalue with multiplicity $2$ and do not support a bifurcation branch, see \ref{exe:ker2deg3nobifbra}.


This text is organized as follows: in \ref{sec:regnet}, we review some concepts and results of networks, coupled cell systems and steady-state bifurcations on networks, focusing on regular networks. 
Choosing a semisimple eigenvalue of the adjacency matrix and considering a generic coupled cell system with a bifurcation condition associated to the eigenvalue,
in \ref{sec:sbl}, we apply the Lyapunov-Schmidt Reduction and a blow-up technique to the coupled cell system reducing the bifurcation problem to the problem of finding zeros of a vector field in a sphere. 
Next, we prove the existence of a bifurcation branch, if the eigenvalue has odd multiplicity, \ref{subsec:oddsbl}. 
Last, we study the existence of bifurcation branches, when the eigenvalue has multiplicity $2$, \ref{subsec:twosbl}.

\section{Settings}\label{sec:regnet}

In this section, we recall some facts about networks and coupled cell systems, following \cite{SGP03, GST05}, and steady-state bifurcations on coupled cell systems.

\subsection{Regular Networks}\label{subsec:regnet}

A \emph{directed graph} is a tuple $G=(C,E,s,t)$, where $c\in C$ is a cell and $e \in E$ is a directed edge from the source cell, $s(e)$, to the target cell, $t(e)$. We assume that the sets of cells and edges are finite.
The \emph{input set} of a cell $c$, $I(c)$, is the set of edges that target $c$.

\begin{defi}
A \emph{regular network} is a directed graph $N$ such that the cardinality of the input set of a cell is the same for all cells. The \emph{valency} $\vartheta$ of $N$ is the number of edges that target each cell. 
We denote the number of cells in $N$ by $|N|$. 
\end{defi}

See \ref{fig:regnet} for two examples of regular networks with valency $2$.

A regular network can be represented by its \emph{adjacency matrix} $A$, where $A$ is a $|N|\times |N|$ matrix and the entry $(A)_{c\:c'}$ is the number of edges from $c'$ to $c$.

\begin{figure}[h]
\center
\begin{subfigure}[t]{0.4\textwidth}
\center
\begin{tikzpicture}
\node (n1) [circle,draw]   {1};
\node (n3) [circle,draw] [right=of n1]  {3};
\node (n4) [circle,draw] [below=of n1]  {4};
\node (n2) [circle,draw]  [below=of n3] {2};

\draw[->, thick] (n4) to [bend right] (n1);
\draw[->, thick] (n4) to [bend right] (n2);
\draw[->, thick] (n2) to (n3);
\draw[->, thick] (n2) to (n4);

\draw[->, thick] (n4) to [bend left] (n1);
\draw[->, thick] (n4) to [bend left] (n2);
\draw[->, thick] (n3) to [loop right] (n3);
\draw[->, thick] (n3) to (n4);
\end{tikzpicture}
\caption{Network \#29 of \cite{K09}}
\label{fig:net29}
\end{subfigure}
\begin{subfigure}[t]{0.4\textwidth}
\centering
\begin{tikzpicture}
\node (n1) [circle,draw]   {1};
\node (n3) [circle,draw] [right=of n1]  {3};
\node (n4) [circle,draw] [below=of n1]  {4};
\node (n2) [circle,draw]  [below=of n3] {2};

\draw[->, thick] (n4) to [bend right] (n1);
\draw[->, thick] (n4) to [bend right] (n2);
\draw[->, thick] (n1) to (n3);
\draw[->, thick] (n1) to (n4);

\draw[->, thick] (n4) to [bend left] (n1);
\draw[->, thick] (n4) to [bend left] (n2);
\draw[->, thick] (n2) to (n3);
\draw[->, thick] (n2) to (n4);
\end{tikzpicture}
\caption{Network \#51 of \cite{K09}}
\label{fig:net51}
\end{subfigure}
\caption{Regular networks with valency $2$.}
\label{fig:regnet}
\end{figure}
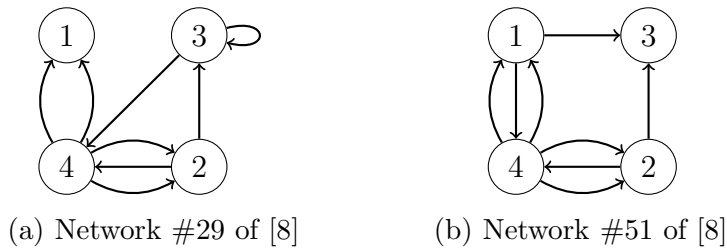

\begin{defi}
A \emph{coloring} of the cells of a network $N$ is an equivalence relation on the set of cells of $N$. The coloring is \emph{balanced} for $N$ if for any two cells of $N$ with the same color there is a bijection between the input sets of the two cells preserving the color of the source cells.
\end{defi}

Any regular network $N$ has two trivial balanced colorings: the \emph{full synchronous coloring}, $\bowtie_0$, and the \emph{full asynchronous coloring} $\bowtie_{=}$. The full synchronous coloring has only one class, i.e., $c\bowtie_0 c'$ for every cells $c,c'$ in $N$, and the full asynchronous coloring has $|N|$ classes, i.e., $c\bowtie_{=}c'$ if and only if $c=c'$.

Each balanced coloring defines a quotient network, \cite[Section 5]{GST05}. 
\begin{defi}
The \emph{quotient network} of a regular network $N$ with respect to a given balanced coloring $\bowtie$ is the network where the equivalence classes of the coloring, $[c]_{\bowtie}$, are the cells and there is an edge from $[c]_{\bowtie}$ to $[c']_{\bowtie}$, for each edge from a cell in the class $[c]_{\bowtie}$ to $c'$. We denote the quotient network by $N/\bowtie$. 
We also say that a network $L$ is a \emph{lift} of $N$, if $N$ is a quotient of $L$ with respect to some balanced coloring of $L$. 
\end{defi}

\begin{figure}[h]
\center
\begin{tikzpicture}
\node (n1) [circle,draw]   {1};
\node (n3) [circle,draw] [right=of n1]  {3};

\draw[->, thick] (n3) to [bend right] (n1);
\draw[->, thick] (n1) to (n3);

\draw[->, thick] (n3) to [bend left] (n1);
\draw[->, thick] (n3) to [loop right] (n3);
\end{tikzpicture}
\caption{Quotient network of the network in \ref{fig:net29} associated to the balanced coloring with classes $\{1,2\}$ and $\{3,4\}$.}
\label{fig:quonet29}
\end{figure}
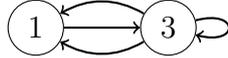

See \ref{fig:quonet29} for an example of a quotient network.

The set of balanced colorings forms a complete lattice, see \cite{S07,K09}. 
Denote by $\Lambda_N$ the set of balanced colorings for $N$. For every $\bowtie_1,\bowtie_2\in\Lambda_N$, we say that $\bowtie_1$ is a \emph{refinement} of $\bowtie_2$, and we write $\bowtie_1\prec \bowtie_2$, if $ \bowtie_1\neq \bowtie_2$ and $c\bowtie_1 d$ implies $c\bowtie_2 d$ for every cells $c,d$ of $N$. We denote by $\preceq$ the relation of refinement or equal. The pair $(\Lambda_N,\preceq)$ forms a lattice. 

Now, we introduce the definition of $\mu$-maximal and $\mu$-submaximal colorings, where $\mu$ is an eigenvalue of the network adjacency matrix.

\begin{defi}
Let $N$ be a regular network, $\bowtie$ a balanced coloring of $N$ and $\mu$ an eigenvalue of the adjacency matrix associated to $N$. We say that $\bowtie$ is a \emph{$\mu$-maximal coloring} if for every $\bowtie'$ such that $\bowtie\prec \bowtie'$ we have that $\mu$ is not an eigenvalue of the adjacency matrix associated to $N/\bowtie'$.
\end{defi}

\begin{defi}
Let $N$ be a regular network, $\bowtie$ a balanced coloring of $N$ and $\mu$ an eigenvalue of the adjacency matrix associated to $N$ with multiplicity $m>1$. We say that $\bowtie$ is a \emph{$\mu$-submaximal coloring of type $j$} if there are $j$ balanced colorings $\bowtie_1, \dots, \bowtie_j$ all distinct such that: $(i)$ $\bowtie\prec \bowtie_i$, $\bowtie_i\not\preceq \bowtie_{i'}$ and $\mu$ is an eigenvalue with  multiplicity $1$ of the adjacency matrix associated to $N/\bowtie_i$ for $i,i'=1,\dots,j$ with $i\neq i'$; $(ii)$ for any other balanced coloring $\bowtie'$ such that $\bowtie\prec \bowtie'$ we have that $\mu$ is not an eigenvalue of the adjacency matrix associated to $N/\bowtie'$ or $\bowtie_i \preceq \bowtie'$ for some $i=1,\dots,j$.
We say that $\bowtie_1, \dots, \bowtie_j$ are the \emph{$\mu$-simple components} of $\bowtie$.
\end{defi}

\begin{exe}\label{exe:latbalcol}
We consider the network \#51 of \cite{K09} (\ref{fig:net51}), which we denote by $N_{51}$ and it has the following adjacency matrix:
$$A_{51}=\left[\begin{matrix}
 0 & 0 & 0 & 2 \\
 0 & 0 & 0 & 2 \\
 1 & 1 & 0 & 0 \\
 1 & 1 & 0 & 0 
\end{matrix}\right].$$
The eigenvalues of $A_{51}$ are: the network valency $2$, $-2$ and $0$ with multiplicity $1$, $1$ and $2$, respectively. The network $N_{51}$ has four non-trivial balanced colorings $\bowtie_1=\{\{1,2\},\{3,4\}\}$, $\bowtie_2=\{\{3\},\{1,2,4\}\}$, $\bowtie_3=\{\{1,2\},\{3\},\{4\}\}$ and $\bowtie_4= \{\{1\},\{2\},\{3,4\}\}$.
 The balanced colorings of $N_{51}$ and the eigenvalues of the adjacency matrix corresponding to the quotient networks of $N_{51}$ associated to each balanced coloring are annotated in \ref{fig:latbalcoloreige51}.

The balanced coloring $\bowtie_0$ is $2$-maximal. The balanced coloring $\bowtie_1$ is $(-2)$-maximal. The balanced colorings $\bowtie_2$ and $\bowtie_4$ are $0$-maximal. And the balanced coloring $\bowtie_=$ is $0$-submaximal of type $2$ with $0$-simple components $\bowtie_3$ and $\bowtie_4$.
\end{exe}

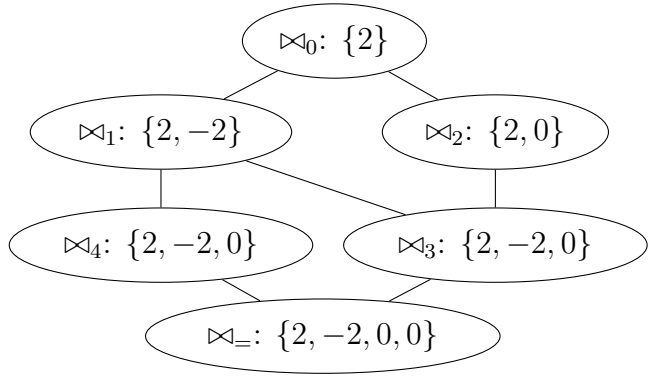
\begin{figure}[h]
\center
\begin{tikzpicture}[node distance=0.5cm and 0.7cm]
\node (n1) [ellipse,draw]  {$\bowtie_{0}$: $\{2\}$};
\node (n2) [ellipse,draw]  [below right=of n1, xshift=-0.5cm] {$\bowtie_2$: $\{2,0\}$};
\node (n3) [ellipse,draw]  [below left=of n1, xshift=0.5cm] {$\bowtie_1$: $\{2,-2\}$};
\node (n4) [ellipse,draw]  [below=of n2] {$\bowtie_3$: $\{2,-2,0\}$};
\node (n5) [ellipse,draw]  [below=of n3] {$\bowtie_4$: $\{2,-2,0\}$};
\node (n6) [ellipse,draw]  [below left=of n4, xshift=1.5cm] {$\bowtie_{=}$: $\{2,-2,0,0\}$};

\draw (n1) to  (n2);
\draw (n1) to  (n3);
\draw (n2) to  (n4);
\draw (n3) to  (n4);
\draw (n3) to  (n5);
\draw (n4) to  (n6);
\draw (n5) to  (n6);
\end{tikzpicture}
\caption{Balanced colorings of $N_{51}$ (\ref{fig:net51}) and the eigenvalues of the adjacency matrices associated to the corresponding quotient networks.}
\label{fig:latbalcoloreige51}
\end{figure}

\subsection{Coupled Cell Systems}\label{subsec:ccs}

In order to associate dynamics to a network, following \cite{SGP03, GST05}, we specify a phase space for the network and describe vector fields that are admissible for the network. 

Let $N$ be a regular network with valency $\vartheta$ and represented by the adjacency matrix $A$. We correspond to each cell $c$ a coordinate $x_c$ and assume that $x_c\in \mathbb{R}$. The \emph{network phase space} is the product of the phase space of the cells, i.e., $\mathbb{R}^{|N|}$.

A vector field $F:\mathbb{R}^{|N|}\rightarrow \mathbb{R}^{|N|}$  is \emph{admissible} for a regular network $N$ if:
\begin{enumerate}
\item The dynamics of cell $c$ depends only on its internal state and on the state of its input cells, $s(I(c))$. Thus there is a function $f:\mathbb{R}\times \mathbb{R}^{\vartheta}\rightarrow \mathbb{R}$ such that for every cell $c$  
$$(F(x))_c=f(x_c, x_{s(I(c))}),$$
where $x_{s(I(c))}=(x_{s(e)})_{e\in I(c)}$;
\item The state of the input cells have equal effect on the dynamics. That is, the function $f$ is $S_\vartheta$-invariant, where  $S_\vartheta$ is the group of permutations in $\{1,\dots,\vartheta\}$. For every $\sigma\in S_\vartheta$
$$f(\sigma(x_0,x_1,\dots,x_\vartheta))=f(x_0,x_1,\dots,x_\vartheta),$$
where $\sigma(x_0,x_1,\dots,x_\vartheta)=(x_0,x_{\sigma(1)},\dots,x_{\sigma(\vartheta)})$.
\end{enumerate}
A \emph{coupled cell system} associated to a regular network $N$ is a dynamical system defined by an admissible vector field $F:\mathbb{R}^{|N|}\rightarrow \mathbb{R}^{|N|}$
$$\dot{x}=F(x),\quad x\in \mathbb{R}^{|N|}.$$

Let $f:\mathbb{R}\times\mathbb{R}^\vartheta\rightarrow\mathbb{R}$ be a $S_\vartheta$-invariant function. We denote by $f^N$ the admissible vector field for $N$ defined by $f$ and given by the previous formulas. Observe that every admissible vector field for $N$ is equal to $f^N$ for some $S_\vartheta$-invariant function $f$.
We say that a function $f:\mathbb{R}\times\mathbb{R}^\vartheta\rightarrow \mathbb{R}$  is \emph{regular} if  $f$ is $S_\vartheta$-invariant and $(0,0,\dots,0)$ is an isolated zero of $f$.
 In this case $0\in \mathbb{R}^{|N|}$ is an equilibrium point of the coupled cell system defined by $f^N$.

For differentiable admissible vector fields $f^N$, its Jacobian at the origin can be represented in terms of the adjacency matrix of $N$, see \cite{LG06}.
We denote by $J_f^N$ the Jacobian of $f^{N}$ at the origin $0\in\mathbb{R}^{|N|}$. We have that
$$J^{N}_f=(D f^{N})_{0}=f_0 Id+f_1 A,$$
where $Id$ is the $|N|\times |N|$ identity matrix,
\begin{equation}\label{eq:f0f1}
f_0=\frac{\partial f}{\partial x_0}(0,0,\dots,0),\  f_1=\frac{\partial f}{\partial x_1}(0,0,\dots,0)=\dots=\frac{\partial f}{\partial x_\vartheta}(0,0,\dots,0).
\end{equation}

For every eigenvalue $\mu$ of $A$ with algebraic multiplicity $m_a$ and geometric multiplicity $m_g$, we have that $f_0+\mu f_1$ is an eigenvalue of $J_f^N$ with the same multiplicities $m_a$ and $m_g$. See \cite[Proposition 3.1]{LG06}. 
Moreover, the kernel of $J_f^N$ can be described using the eigenvectors of $A$. 
Denote by  $v_1,\dots,v_j$ a set of linear independent eigenvectors of $A$ and $\mu_1,\dots,\mu_j$ its corresponding eigenvalues, where $j$ is equal to the sum of all geometric multiplicities. Then
	$$\ker(J_f^N)=\{v: J_f^N v=0\}=\spn(\{v_i:f_0+\mu_i f_1=0\}),$$
	where $\spn$ denotes the linear subspace spanned by the vectors.

\begin{defi}
A \emph{polydiagonal subspace} of $\mathbb{R}^{|N|}$ is a subspace defined by the equality of certain cell coordinates. Let $\bowtie$ be a coloring in $N$. The polydiagonal subspace associated to $\bowtie$ is defined by
$$\Delta_{\bowtie}=\left\{x: x_c=x_d \Leftarrow c\bowtie d \right\}\subseteq \mathbb{R}^{|N|}.$$
Each polydiagonal subspace defines an unique coloring of the cells.
A subset $K\subseteq V$ is \emph{invariant} under a map $g:V\rightarrow V$, if $g(K)\subseteq K$.  A \emph{synchrony subspace} of a network is an invariant polydiagonal subspace for any vector field admissible for the network. 
\end{defi}

We have that the synchrony subspaces and balanced colorings are in one-to-one correspondence. 

\begin{teo}[{\cite[Theorem 4.3]{GST05}}]
Let $\bowtie$ be a coloring of cells in a network $N$. 
Then $\Delta_{\bowtie}$ is a synchrony subspace of $N$ if and only if $\bowtie$ is balanced.
\end{teo}

The restriction of an admissible vector field to a synchrony subspace $\Delta_{\bowtie}$ is an admissible vector field for the quotient network associated to the balanced coloring $\bowtie$. Moreover, any admissible vector field for the quotient network lifts to an admissible vector field for the network.

\begin{teo}[{\cite[Theorem 5.2]{GST05}}]\label{GST0552}
Let $N$ be a regular network with valency $\vartheta$, $\bowtie$ a balanced coloring of $N$ and $f:\mathbb{R}\times\mathbb{R}^\vartheta\rightarrow\mathbb{R}$ a $S_\vartheta$-invariant function. If $Q$ is the quotient network of $N$ associated to $\bowtie$, then

\noindent
$(i)$ The restriction of $f^N$ to $\Delta_{\bowtie}$ is the admissible vector field $f^Q$ for $Q$.

\noindent
$(ii)$ The admissible vector field $f^Q$ for $Q$ lifts to the admissible vector field $f^N$ for $N$.
\end{teo}

In particular, the previous result means that if $x_Q(t)\in\mathbb{R}^{|Q|}$ is a solution to $\dot{x}_Q(t)=f^{Q}(x_Q(t))$, then $x_N(t)$ is a solution to $\dot{x}_N(t)=f^{N}(x_N(t))$, where $(x_N(t))_c=(x_Q(t))_{[c]_{\bowtie}}$ for each cell $c$ in $N$. And we say that $x_Q(t)$ is \emph{lifted} to $x_N(t)$.

\subsection{Steady-Sate Bifurcation on Regular Networks}\label{subsec:bifregnet}

Let $N$ be a regular network with valency $\vartheta$ and represented by the adjacency matrix $A$. We consider a family of regular functions $f:\mathbb{R}\times \mathbb{R}^{\vartheta}\times \mathbb{R}\rightarrow \mathbb{R}$ depending on a parameter $\lambda$, i.e., for each $\lambda\in \mathbb{R}$, the function $f(x_0,x_1,\dots,x_{\vartheta},\lambda)$ is  regular. In this work, we assume that this function is smooth in some neighborhood of the origin. Consider the coupled cell system 
\begin{equation}\label{eq:ccs}
\dot{x}=f^{N}(x,\lambda),
\end{equation}
where $f^{N}:\mathbb{R}^{|N|}\times \mathbb{R}\rightarrow \mathbb{R}^{|N|}$ is given for each cell $c$ in $N$ by
$$(f^{N}(x,\lambda))_c=f(x_c,x_{s(I(c))},\lambda).$$ 
Since we are assuming that $f$ is regular, the origin $0\in \mathbb{R}^{|N|}$ is an equilibrium point of the coupled cell system for every $\lambda\in \mathbb{R}$.

We are interested in studying the steady-state bifurcations of \ref{eq:ccs} occurring from the origin $x=0$ at $\lambda=0$. A local steady-state bifurcation at $\lambda=0$ near the origin can only occur if $J_f^N$ has a zero eigenvalue. Thus we will assume that one of the eigenvalues of $J_f^N$ is zero, say, $f_0+\mu f_1=0$ for some eigenvalue $\mu$ of $A$, where $f_0$ and $f_1$ are defined in \ref{eq:f0f1}. Observe that, under the generic hypothesis on $f$, $f_1\neq0$, $\mu$ is the unique eigenvalue satisfying $f_0+\mu f_1=0$. In this case, we say that $f$ is a \emph{regular function with a bifurcation condition associated to $\mu$}.

\begin{defi}
Let $N$ be a regular network with valency $\vartheta$ and represented by the adjacency matrix $A$, and $f:\mathbb{R}\times \mathbb{R}^{\vartheta}\times \mathbb{R}\rightarrow \mathbb{R}$ a regular function. We say that a differentiable function  $b=(b_N,b_{\lambda}):[0,\delta[\rightarrow \mathbb{R}^{|N|}\times\mathbb{R}$ is an \emph{equilibrium branch of $f$ on $N$} if $b(0)=(0,\dots,0,0)$, $b_{\lambda}(z)\neq 0$ and $$f^N(b_N(z),b_{\lambda}(z))=0,$$ for every $z>0$.
We say that an equilibrium branch $b$ is \emph{a bifurcation branch of $f$ on $N$} if $b$ is different from the trivial equilibrium branch of $f$ on $N$, i.e.,  for every $z\neq 0$,  \[b_N(z)\neq 0.\]
\end{defi}

Despite two different bifurcation branches can define essentially the same branch (e.g., by rescaling of the parameter), it is not a problem for our discussion about the existence of bifurcation branches (see \cite[Section 4.2]{F07} for a definition of  bifurcation branch that takes this aspect into account).

The trivial equilibrium branch is totally synchronized, since all cell's coordinates have the same value. Other bifurcation branches can have less synchrony depending on which synchrony subspaces they belong.

\begin{defi}
We say that an equilibrium branch $b:[0,\delta[\rightarrow \mathbb{R}^{|N|}\times\mathbb{R}$ has \emph{(exactly) synchrony $\bowtie$}, if $b_N([0,\delta[)\subseteq \Delta_{\bowtie}$ (and  $b_N([0,\delta[)\not\subseteq \Delta_{\bowtie'}$ for every $\bowtie'$ such that $\bowtie\prec \bowtie'$).
\end{defi}

In the same way we lift solutions, we can lift bifurcation branches on a quotient network to the original network, see the end of \ref{subsec:ccs}.

\section{Synchrony Branching Lemma}\label{sec:sbl}

In this section, we establish the two main steps in order to prove the existence of a bifurcation branch of generic regular functions on regular networks. First we apply the method of Lyapunov-Schmidt Reduction, \cite[Chapter VII]{GS85}, to the coupled cell system and we obtain a reduced equation for the bifurcation problem. Next we apply a blow-up argument (see e.g. \cite{LLH94}) to transform the reduced equation into a vector field on a sphere. 

Let $N$ be a regular network with valency $\vartheta$ and represented by the adjacency matrix $A$, $\mu$ an eigenvalue of $A$ and $f:\mathbb{R}\times\mathbb{R}^\vartheta\times\mathbb{R}\rightarrow \mathbb{R}$ a regular function with a bifurcation condition associated to $\mu$. 
Hence $\ker(J_f^N)\neq \{0\}$ and we assume the generic hypothesis that $m=\dim(\ker(J_f^N))>0$ is the geometric multiplicity of $\mu$ in $A$. 

Let $v_1,\dots, v_m$ be a basis for $\ker(J_f^N)$ (and eigenvectors of $A$ associated to $\mu$), $v^*_1,\dots,v^*_m$ be a basis for $\ran( J_f^N)^{\bot}$. Applying  the Lyapunov-Schmidt Reduction method, we get a function $g:\mathbb{R}^m\times\mathbb{R}\rightarrow\mathbb{R}^m$ such that the solutions of $f^N(x,\lambda)=0$ are in one-to-one correspondence with the solutions of $g(y,\lambda)=0$. 
We can calculate the derivatives of $g$ at the origin using the derivatives of $f$ at the origin, see \cite[Chapter VII \S 1 (d)]{GS85}. Since $f(0,0,\dots,0,\lambda)=0$ for every $\lambda$, we have that 
$$g(0,\lambda)=0,\quad\quad \frac{\partial g_i}{\partial y_j}(0,0)=0,\quad \quad \frac{\partial g_i}{\partial \lambda}(0,0)=0.$$
The Taylor expansion of $g$ at $(y,\lambda)=(0,0)$ has the following form:
$$g(y,\lambda)=L(\lambda)y+Q_k(y)+\mathcal{O}(\|y\|^{k+1}+\|y\|^2|\lambda|),$$
where 
$$L(\lambda)=\lambda Dg_{\lambda}+\mathcal{O}(2),$$
$Dg_\lambda$ is the matrix with entries $(\partial^2 g_i/\partial \lambda \partial y_j)$ evaluated at $(y,\lambda)=(0,0)$, $Q_k$ has homogenous polynomial components in the variable $y$ of smallest degree $k$ such that $Q_k$ does not vanish.
From \cite[Chapter VII \S 1 (d)]{GS85}, we know that
$$\frac{\partial^2 g_i}{\partial y_j\partial \lambda}(0,0)= \langle v_i^*, (D f_{\lambda}^N) v_j\rangle =\langle v_i^*, (f_{0\lambda}Id + f_{1\lambda} A) v_j\rangle=(f_{0\lambda} + \mu f_{1\lambda} ) \langle v_i^*, v_j\rangle,$$
where $\langle\cdot,\cdot\rangle$ is the usual inner product in $\mathbb{R}^{|N|}$. Therefore
$$D g_{\lambda}= (f_{0\lambda}+\mu f_{1\lambda})L,$$
where $$L=\left[\begin{matrix}
\langle v^*_1 , v_1\rangle&\langle v^*_1 , v_2\rangle &\dots & \langle v^*_1 , v_m\rangle\\
\langle v^*_2 , v_1\rangle&\langle v^*_2 , v_2\rangle &\dots & \langle v^*_2 , v_m\rangle\\
\vdots&\vdots &\ddots & \vdots \\
\langle v^*_m , v_1\rangle&\langle v^*_m , v_2\rangle &\dots & \langle v^*_m , v_m\rangle
\end{matrix}\right].$$

We will assume that the eigenvalue $\mu$ is \emph{semisimple}, i.e., $\mu$ has the same algebraic and geometric multiplicity. Note that $L$ is invertible 
 if and only if $\mu$ is semisimple.

Since $L$ is invertible, we can choose a basis $v'^*_1,\dots,v'^*_m$ of $\ran(J_f^N)^{\bot}$ such that 
$$\left[\begin{matrix}
\langle v'^*_1 , v_1\rangle&\langle v'^*_1 , v_2\rangle &\dots & \langle v'^*_1 , v_m\rangle\\
\langle v'^*_2 , v_1\rangle&\langle v'^*_2 , v_2\rangle &\dots & \langle v'^*_2 , v_m\rangle\\
\vdots&\vdots &\ddots & \vdots \\
\langle v'^*_m , v_1\rangle&\langle v'^*_m , v_2\rangle &\dots & \langle v'^*_m , v_m\rangle
\end{matrix}\right]=Id,$$
by taking $v'^*_i=\sum_{l=1}^m b_{i\:l} v^*_l$ where $b_{i\:j}$ are the entries of $L^{-1}$ for $1\leq i,j \leq m$.

In the following we assume that $\mu$ is semisimple and that we have chosen a basis of $\ran(J_f^N)^{\bot}$ in the Lyapunov-Schmidt Reduction such
\begin{equation}
g(y,\lambda)=L(\lambda) y +Q_k(y)+\mathcal{O}(\|y\|^{k+1}+\|y\|^2|\lambda|),
\label{eq:taylypred}
\end{equation}
where $Q_k$ has homogenous polynomial components in the variable $y$ of smallest degree $k$ such that $Q_k$ does not vanish and
$$L(\lambda)=\lambda(f_{0\lambda}+\mu f_{1\lambda})Id+\mathcal{O}(2).$$

\begin{defi}
Let $N$ be a regular network with valency $\vartheta$ and $f:\mathbb{R}\times\mathbb{R}^\vartheta\times\mathbb{R}\rightarrow \mathbb{R}$ a regular function. We denote by $k(N,f)$ the integer $k$ in \ref{eq:taylypred}. We say that the bifurcation problem of $f$ on $N$ has $k-1$ \emph{degeneracy}.
\end{defi}

In \cite{SG11}, the authors studied the degeneracy of a bifurcation problem on regular networks associated to simple eigenvalues. They have shown that there exist bifurcation problems on regular networks with high degeneracy. We refer the reader to their work for examples of $k$-degenerated bifurcation problems on regular networks, with $1 \leq k\leq5$.

Before we prove that the integer $k(N,f)$ does not depend generically on the regular function $f$ associated to some eigenvalue, we give  an explicit formula for the second derivatives of $g_i$ with respect to $y_j$ and $y_l$ for $1\leq i,j,l\leq m$.

Since $f:\mathbb{R}\times\mathbb{R}^\vartheta\times\mathbb{R}\rightarrow \mathbb{R}$ is a regular function, it has the following Taylor expansion at the origin:
\begin{align*}
f(x_0,x_1,\dots,x_\vartheta,\lambda)&=f_0 x_0 + f_1(x_1+\dots+x_\vartheta)+ f_{0\lambda} x_0 \lambda +  f_{1\lambda}(x_1+\dots+x_\vartheta)\lambda \\
&+ \frac{f_{00}}{2} x_0^2 + f_{01} x_0(x_1+\dots + x_\vartheta) + \frac{f_{11}}{2} (x_1^2+\dots +x_\vartheta^2)  \\
&+ f_{1\vartheta}(\sum_{i=1}^\vartheta\sum_{j>i}^\vartheta x_{i}x_j)+ \mathcal{O}(3),
\end{align*}
where 
\begin{align*}
f_{00}= \frac{\partial^2 f}{\partial x_0 \partial x_0}(0,0),&\quad &  f_{11}= \frac{\partial^2 f}{\partial x_1 \partial x_1}(0,0)=\frac{\partial^2 f}{\partial x_i \partial x_i}(0,0),\\
f_{01}= \frac{\partial^2 f}{\partial x_0 \partial x_1}(0,0)= \frac{\partial^2 f}{\partial x_0 \partial x_i}(0,0), & \quad  & f_{1\vartheta}= \frac{\partial^2 f}{\partial x_1 \partial x_\vartheta}(0,0)= \frac{\partial^2 f}{\partial x_i \partial x_j}(0,0),
\end{align*}
for $i,j>0$ and $i\neq j$. 

For every cell $c$ in $N$, we denote by $c_1,\dots,c_\vartheta$ the source cells of the edges that target $c$ (repeated, if there is more than one edge from the same cell).
\begin{align*}
d^2 f_c^N(v_j&,v_l)= f_{00}(v_j\ast v_l)_c+ 2 f_{01} (v_j\ast A v_l)_c+ \sum_{a=1}^\vartheta \sum_{b=1}^\vartheta f_{ab}(v_j)_{c_a}(v_l)_{c_b}\\
=&(f_{00}+2\mu f_{01})(v_j\ast v_l)_c + f_{1v}(Av_j\ast Av_l)_c + (f_{11}-f_{1v}) \sum_{a=1}^\vartheta (v_j\ast v_l)_{c_a}\\
=&(f_{00}+2\mu f_{01}+\mu^2 f_{1v})(v_j\ast v_l)_c + (f_{11}-f_{1v}) (A(v_j\ast v_l))_c,
\end{align*}
where $w\ast z=(w_1 z_1,w_2 z_2,\dots,w_n z_n)$, for $w,z\in \mathbb{R}^n$. So
$$d^2 f^N(v_j,v_l)=(f_{00}+2\mu f_{01}+\mu^2 f_{1v})v_j\ast v_l + (f_{11}-f_{1v}) A(v_j\ast v_l).$$

It follows from \cite[Chapter VII \S 1 (d)]{GS85} and the Taylor expansion of $f$ that
\begin{align*}
\frac{\partial^2 g_i}{\partial y_j\partial y_l}(0,0)=
(f_{00}+2\mu f_{01}+\mu^2 f_{1\vartheta})\langle v_i^*,v_j\ast v_l\rangle + (f_{11}-f_{1\vartheta})  \langle v_i^*, A(v_j\ast v_l) \rangle,
\end{align*}
for $1\leq i,j,l\leq m$. Since $\mu$ is semisimple, $v_i^*$ is orthogonal to any generalized eigenvector associated to an eigenvalue different from $\mu$. Writing  $v_j\ast v_l$ in the base of generalized eigenvectors of $A$, we have that $\langle v_i^*, A(v_j\ast v_l) \rangle=\mu \langle v_i^*,v_j\ast v_l\rangle$. Thus, for $1\leq i,j,l\leq m$,
$$\frac{\partial^2 g_i}{\partial y_j\partial y_l}(0,0)=(f_{00}+2\mu f_{01}+\mu f_{11}-\mu f_{1\vartheta}+\mu^2 f_{1\vartheta})\langle v_i^*,v_j\ast v_l\rangle.
$$
Thus, generically, the vanish of the second derivatives of $g$ is independent of the regular function $f$ with a bifurcation condition associated to $\mu$.
Next, we prove that the smallest integer $k(N,f)$ is, generically, the same for every regular function $f$ with a bifurcation condition associated to $\mu$.

\begin{lem}\label{lem:degnetstr}
Let $N$ be a regular network with valency $\vartheta$ and adjacency matrix $A$, $\mu$ an eigenvalue of $A$ and $f,f':\mathbb{R}\times\mathbb{R}^\vartheta\times\mathbb{R}\rightarrow \mathbb{R}$ regular functions with a bifurcation condition associated to $\mu$. Then, generically,
$$k(N,f)=k(N,f').$$
\end{lem}

\begin{proof}
For a given integer $l$, we can rearrange the terms in the Taylor expansion of $f$ as follows
\begin{align*}
 f(x_0,x_1,\dots,x_\vartheta,\lambda)&=P_{l-1}(x_0,x_1,\dots,x_\vartheta)+  P_{l}(x_0,x_1,\dots,x_\vartheta)\\ &+P_{l+1}(x_0,x_1,\dots,x_\vartheta)+R(x_0,x_1,\dots,x_\vartheta,\lambda),\end{align*}
where $P_{l-1}$ is a polynomial of degree lower or equal to $l-1$, $P_{l+1}$ has only terms of degree upper or equal to $l+1$, $R$ is a function such that $R(x_0,x_1,\dots,x_\vartheta,0)=0$ and $P_l$ is homogenous polynomial of degree $l$:
$$P_{l}(x_0,x_1,\dots,x_\vartheta)=\sum_{\substack{0\leq n_0,n_1,\dots,n_\vartheta\leq l \\ n_0+n_1+\dots+n_\vartheta=l}} \frac{f^{n_0 n_1 \dots n_\vartheta}}{n_0 !n_1 !\dots n_\vartheta !} x_0^{n_0}x_1^{n_1} \dots x_\vartheta^{n_\vartheta},$$
$f^{n_0 n_1 \dots n_\vartheta}$ is the $l$-partial derivative of $f$ at $(0,0,\dots,0,0)$ with respect $n_0$ times to $x_0$, $n_1$ times to $x_1$, $\dots$, and $n_\vartheta$ times to $x_\vartheta$.
 Since $f$ is $S_\vartheta$-invariant, $P_{l}$ is also $S_\vartheta$-invariant and it has the following form, see e.g. \cite{M95}, 
$$P_{l}(x_0,x_1,\dots,x_\vartheta)=\sum_{n_0=0}^{l}\sum_{\substack{0\leq  n_1 \leq \dots\leq n_\vartheta\leq l \\ n_0+n_1+\dots+n_\vartheta=l}} \frac{f^{n_0 n_1 \dots n_\vartheta}}{n_0 !n_1 !\dots n_\vartheta !} x_0^{n_0} \left( \sum_{\sigma\in S_{\vartheta}} x_{\sigma(1)}^{n_1} \dots x_{\sigma(\vartheta)}^{n_\vartheta}\right).$$
For $1\leq i,i_1,\dots,i_l\leq m$, the $l$-th derivative of $g_i$ with respect to $y_{i_1}, y_{i_2},\dots, y_{i_l}$ at $(y,\lambda)=(0,\dots,0,0)$ is given by
$$\frac{\partial^{l}g_i}{\partial y_{i_1}\partial y_{i_2}\dots\partial y_{i_l}}= \sum_{n_0=0}^{l}\sum_{\substack{0\leq  n_1 \leq \dots\leq n_\vartheta\leq l \\ n_0+n_1+\dots+n_\vartheta=l}} \frac{f^{n_0 n_1 \dots n_\vartheta}}{n_0 !n_1 !\dots n_\vartheta !} \langle v^*_i,A_{n_0 n_1 \dots n_\vartheta}(v_{i_1},\dots,v_{i_l})\rangle,$$
where $A_{n_0 n_1 \dots n_\vartheta}(v_{i_1},\dots,v_{i_l})\in \mathbb{R}^{|N|}$ and it is given for every cell $c$ in $N$ by
$$(A_{n_0 n_1 \dots n_\vartheta}(v_{i_1},\dots,v_{i_l}))_c=
\sum_{\sigma\in S_{\vartheta}} \frac{\partial^l}{\partial t_1\dots \partial t_l} \left.\left( \prod_{b=0}^{\vartheta}(t_{1}v_{i_1}+\dots+t_l v_{i_l})_{c_{\sigma(b)}}^{n_b} \right)\right|_{t_i=0},$$
where $\sigma(0)=0$, $c_0=c$ and $c_1,\dots,c_\vartheta$ are the source cells of the edges that target $c$.

 Note that every term in the variable $y$ of degree $l$ in $g$ vanish if and only if for every $1\leq i,i_1,\dots,i_l\leq m$
$$\frac{\partial^{l}g_i}{\partial y_{i_1}\partial y_{i_2}\dots\partial y_{i_l}}=0.$$

For each $1\leq i,i_1,\dots,i_l\leq m$, regard $(\partial^{l}g_i)/(\partial y_{i_1}\partial y_{i_2}\dots\partial y_{i_l})$ as a polynomial function in the variables $f^{n_0 n_1 \dots n_\vartheta}$, where $0\leq n_0\leq l$ and $0\leq  n_1\leq n_2 \leq \dots\leq n_\vartheta\leq l$ such that $n_0+n_1+\dots+n_\vartheta=l$. We have two cases: $(\partial^{l}g_i)/(\partial y_{i_1}\partial y_{i_2}\dots\partial y_{i_l})$ is identically zero since 
$$\langle v^*_i,A_{n_0 n_1 \dots n_\vartheta}(v_{i_1},v_{i_2},\dots,v_{i_l})\rangle=0,$$
for every $0\leq n_0\leq l$ and $0\leq  n_1\leq n_2 \leq \dots\leq n_\vartheta\leq l$; otherwise the set of regular functions such that $(\partial^{l}g_i)/(\partial y_{i_1}\partial y_{i_2}\dots\partial y_{i_l})=0$ is residual and for every generic regular function  we have that
$$\frac{\partial^{l}g_i}{\partial y_{i_1}\partial y_{i_2}\dots\partial y_{i_l}}\neq 0.$$

Therefore, generically, given $f$ regular every term in the variable $y$ of degree $l$ in $g$ vanishes if and only if 
$$\langle v^*_i,A_{n_0 n_1 \dots n_\vartheta}(v_{i_1},v_{i_2},\dots,v_{i_l})\rangle=0,$$
for every $1\leq i,i_1,\dots,i_l\leq m$, $0\leq n_0\leq j$ and $0\leq  n_1\leq n_2 \leq \dots\leq n_\vartheta\leq j$.

The second part of the previous ``if and only if'' does not depend on the regular function $f$. So for every regular functions $f,f':\mathbb{R}\times\mathbb{R}^\vartheta\times\mathbb{R}\rightarrow \mathbb{R}$ with a bifurcation condition associated to $\mu$, we have  generically that
\[k(N,f)=k(N,f').\]
\end{proof}

Following the previous lemma, we define the smallest $k$ in \ref{eq:taylypred} as a function of the network $N$ and the eigenvalue $\mu$.

\begin{defi}
Let $N$ be a regular network and $\mu$ an eigenvalue of its adjacency matrix. For any generic regular function $f$ with a bifurcation condition associated to $\mu$, we define $k(N,\mu)= k(N,f)$. We say that the bifurcation problem associated to $\mu$ on $N$ has $k(N,\mu)-1$ \emph{degeneracy}.
\end{defi}

\begin{obs}
Let $N$ be a regular network, $Q$ a quotient network of $N$ and $\mu$ an eigenvalue of the adjacency matrix associated to $Q$. Then $$k(N,\mu)\leq k(Q,\mu).$$
In fact, let $f$ be a regular function with a bifurcation condition associated to $\mu$. Denote by $g^Q$ and $g^N$ the reduced functions of $f^Q$ and $f^N$, respectively, obtained by Lyapunov-Schmidt reduction. The previous inequality follows from the fact $$g^Q=P^{-1} g^N P,$$ 
where $P: \mathbb{R}^{|Q|}\rightarrow \Delta\subseteq \mathbb{R}^{|N|}$ is the lift of the phase space of $Q$ into the synchrony subspace of $N$ associated to the quotient network and $P^{-1}: \Delta\rightarrow\mathbb{R}^{|Q|}$ is the inverse of $P$.
\end{obs}

In the next step we apply a blow-up argument also used in the equivariant theory, see e.g. \cite{LLH94}. Applying the following change of variables to the reduced function $g$ of the Lyapunov-Schmidt reduction \ref{eq:taylypred},
\begin{equation}\label{eq:blowuptrans}
\begin{cases}
y=\epsilon u\\
\lambda= \epsilon^{k-1} \eta
\end{cases},
\end{equation}
where $u\in S^{m-1}$($m-1$ dimensional sphere), $\epsilon\in\mathbb{R}$, $\eta\in \mathbb{R}$ and $k=k(N,\mu)$, we have the following equation:
$$g(y,\lambda)=0 \Leftrightarrow g(\epsilon u, \epsilon^{k-1} \eta)=0 \Leftrightarrow \epsilon^{k}\left( \eta(f_{0\lambda}+\mu f_{1\lambda}) u + Q_k(u)+\mathcal{O}(|\epsilon|)\right)=0.$$

Let $h:S^{m-1}\times\mathbb{R}\times \mathbb{R}\rightarrow \mathbb{R}^{m}$ be the function given by $$h(u,\epsilon,\eta)=\eta(f_{0\lambda}+\mu f_{1\lambda}) u + Q_k(u)+\mathcal{O}(|\epsilon|).$$ For $y\neq 0$, we have that
\begin{equation}\label{eq:blowupofg}
g(y,\lambda)=0 \Leftrightarrow h(u,\epsilon,\eta)=0.
\end{equation}

\begin{prop}\label{prop:bifbrtozero}
Let $N$ be a regular network with valency $\vartheta$ and adjacency matrix $A$, $\mu$ a semisimple eigenvalue of $A$ with multiplicity $m$ and $f:\mathbb{R}\times\mathbb{R}^\vartheta\times\mathbb{R}\rightarrow \mathbb{R}$ a regular function with a bifurcation condition associated to $\mu$. If $b$ is a bifurcation branch of $f$ on $N$ with synchrony $\bowtie$, then generically there exists $\tilde{u}\in S^{m-1}$ and $\tilde{\eta}\in\mathbb{R}$ such that
$$h(\tilde{u},0,\tilde{\eta})=0,$$
and $$\tilde{u}_1 v_1 +\dots+ \tilde{u}_m v_m \in \Delta_{\bowtie},$$
where $v_1,\dots, v_m$ is the basis of $\ker(J_f^N)$.
\end{prop}

\begin{proof}
Let $b=(b_N,b_\lambda)$ be a bifurcation branch of $f$ on $N$ and $k=k(N,\mu)$. Consider $\tilde{b}=(\tilde{b}_y,\tilde{b}_\lambda):[0,\delta[\rightarrow \mathbb{R}^m\times \mathbb{R}$ such that $\tilde{b}_y$ is the projection of $b_N$ into $\ker(J_f^N)$ according to the basis $v_1,\dots, v_m$ and $\tilde{b}_\lambda=b_\lambda$. Then
$$g(\tilde{b}_y(z),\tilde{b}_\lambda(z))=h\left(\frac{\tilde{b}_y(z)}{\|\tilde{b}_y(z)\|},\|\tilde{b}_y(z)\| ,\frac{\tilde{b}_\lambda(z)}{\|\tilde{b}_y(z)\|^{k-1}}\right)=0,$$
for $z\neq 0$.
Note that $\tilde{b}'_y(0)\neq 0$, see e.g. \cite[Lemma 4.2.1]{F07}. We can also prove by induction on $j$, that $\tilde{b}^{(j)}_{\lambda}(0)= 0$, for $1 \leq j<k-1$, where $\tilde{b}^{(j)}_{\lambda}(0)$ is the $j$-derivative of $\tilde{b}_{\lambda}$ at $z=0$.
Taking the limit of $z\rightarrow 0$ in the previous equation,
$$h(\tilde{u},0,\tilde{\eta})=0,$$
where $$\tilde{u}= \frac{\tilde{b}'_y(0)}{\|\tilde{b}'_y(0)\|},\quad \tilde{\eta}=\frac{\tilde{b}^{(k-1)}_\lambda(0)}{(k-1)!\|\tilde{b}'_y(0)\|^{k-1}}.$$ 

If $b$ has synchrony $\bowtie$, then $b_N(z)\in\Delta_{\bowtie}$ and  $\tilde{b}_{y_1}(z) v_1 +\dots+ \tilde{b}_{y_m}(z) v_m \in \Delta_{\bowtie}$ for every $z\geq 0$. So $\tilde{b}'_{y_1}(0) v_1 +\dots+ \tilde{b}'_{y_m}(0) v_m \in \Delta_{\bowtie}$ and 
\[\tilde{u}_1 v_1 +\dots+ \tilde{u}_m v_m \in \Delta_{\bowtie}.\]
\end{proof}

It follows from \ref{prop:bifbrtozero} that we can study the zeros of $h(u,0,\eta)$ to understand the bifurcation branches. For the radial component 
$$ \langle h(u,0,\eta),u\rangle=0 \Leftrightarrow \eta= - \frac{\langle Q_k(u), u\rangle}{f_{0\lambda}+\mu f_{1\lambda}},$$
where it is assumed, generically, that $f_{0\lambda}+\mu f_{1\lambda}\neq0$.
Defining $$\tilde{\eta}(u)=- \frac{\langle Q_k(u), u\rangle}{f_{0\lambda}+\mu f_{1\lambda}}, \quad \quad\quad\quad \tilde{h}(u)= h(u,0,\tilde{\eta}(u)),$$ we obtain a vector field $\tilde{h}$ on the $(m-1)$-sphere and 
$$h(u,0,\eta)=0 \Leftrightarrow \tilde{h}(u)=0 \wedge \eta(u)= - \frac{\langle Q_k(u), u\rangle}{f_{0\lambda}+\mu f_{1\lambda}}.$$

Now, we see how to make the correspondence between a zero of the vector field $\tilde{h}$ and a bifurcation branch of $f$ on $N$. For this purpose, we will assume for a zero $\tilde{u}\in S^{m-1}$ of $\tilde{h}$ that
\begin{equation}\label{eq:H2}
\left(\frac{\partial h}{\partial (u,\eta)}\right)_{(\tilde{u},0,\tilde{\eta}(\tilde{u}))} \textrm{ is non-singular},
\end{equation}
where the Jacobian is calculated in the geometry of $S^{m-1}\times\mathbb{R}$. 

Moreover, we say that condition \ref{eq:H1} holds for $f$ and $N$ if
\begin{equation}
\label{eq:H1}
\tag{H1}
\forall_{\tilde{u}}\ \  \tilde{h}(\tilde{u})=0 \Rightarrow \ref{eq:H2} \textrm{ holds for } \tilde{u}.
\end{equation}

\begin{prop}\label{prop:bifbrfromzero}
Let $N$ be a regular network with valency $\vartheta$ and adjacency matrix $A$, $\mu$ a semisimple eigenvalue of $A$ with multiplicity $m$, $f:\mathbb{R}\times\mathbb{R}^\vartheta\times\mathbb{R}\rightarrow \mathbb{R}$ a regular function with a bifurcation condition associated to $\mu$ and $\tilde{u}\in S^{m-1}$ such that $\tilde{h}(\tilde{u})=0$ and \ref{eq:H2} holds for $\tilde{u}$. Then generically there exists a bifurcation branch of $f$ on $N$.
\end{prop}
 
\begin{proof}
We have that $h(\tilde{u},0,\tilde{\eta}(\tilde{u}))=0$, because $\tilde{h}(\tilde{u})=0$.
Since \ref{eq:H2} holds for $\tilde{u}$, it follows from the implicit function theorem that there exists a neighborhood $W\subseteq \mathbb{R}$ of $\epsilon=0$ and a differentiable function $(u^*,\eta^*):W\rightarrow S^{m-1}\times \mathbb{R}$ such that $\tilde{h}(u^*(\epsilon),\epsilon,\eta^*(\epsilon))=0$ and $(u^*,\eta^*)(0)=(\tilde{u},\tilde{\eta})$. Recalling \ref{eq:blowuptrans} and \ref{eq:blowupofg}, we have that
$$h(u^*(\epsilon),\epsilon,\eta^*(\epsilon))=0 \Leftrightarrow g(\epsilon u^*(\epsilon),\epsilon^{k-1}\eta^*(\epsilon))=0 \Leftrightarrow g(y^*(\epsilon),\lambda^*(\epsilon))=0,$$ 
defining $(y^*,\lambda^*): W\rightarrow \mathbb{R}^{m}\times \mathbb{R}$ by $y^*(\epsilon)=\epsilon u^*(\epsilon)$ and $\lambda^*(\epsilon)=\epsilon^{k-1}\eta^*(\epsilon)$.

By the Lyapunov-Schmidt reduction, there exists a differentiable function $b^*:W\rightarrow \mathbb{R}^{|N|}\times\mathbb{R}$ associated to $(y^*,\lambda^*)$ such that $0\in W$, $b^*(0)=(0,0)$, $b^*_N(\epsilon)\neq 0$ and $f^N(b_N^*(\epsilon),b_\lambda^*(\epsilon))=0$, for every $\epsilon \in W\setminus \{0\}$. Since the origin is an isolated zero of $f^N(x,0)$, there exists $\delta>0$ such that $b_\lambda^*(\epsilon)\neq 0$ for $0<\epsilon< \delta$. Restricting the function $b^*$ to $[0,\delta[$, we have that $b^*$ is a bifurcation branch of $f$ on $N$. 
\end{proof}

\begin{obs}
\ref{eq:H2} is related to the determinacy of a bifurcation problem. When the kernel of the Jacobian is one-dimensional, the determinacy and degeneracy are equal. If $m=1$ and $f_{0\lambda}+\mu f_{1\lambda}\neq 0$, then \ref{eq:H2} holds.
When $m>1$, determinacy and degeneracy can be different, see \ref{exe:latbalcoldegdet}. 
\end{obs}

If condition~\ref{eq:H2} fails, we can apply the Lyapunov-Schmidt Reduction to the function $h(u,\epsilon,\eta)$ at $(\tilde{u},0,\tilde{\eta})$ and a blow-up change of coordinates to obtain a vector field on a sphere, then we look for zeros of this new reduced vector field on a sphere associated to the parameter $\epsilon$. Since the Jacobian of $h$ with respect to $(u,\eta)$ at $(\tilde{u},0,\tilde{\eta}(\tilde{u}))$ is not identically null, the dimension of the problem will be further reduced and we will need to calculate derivatives of higher order of the original vector field $f^N$. In the new reduced vector field on a sphere, $\tilde{u}$ can continue or not to be an equilibrium. If it is not an equilibrium, then do not correspond to a bifurcation branch. If it continues to be an equilibrium and a similar condition to \ref{eq:H2} holds, then we obtain a bifurcation branch. If it continues to be an equilibrium and a similar condition to \ref{eq:H2} fails, we need to repeat the previous process. 
See e.g. \cite{RK87}.

If a zero $\tilde{u}$ of $\tilde{h}$ corresponds to a point in some synchrony subspace $\Delta_{\bowtie}$ ($\tilde{u}_1 v_1 +\dots+ \tilde{u}_m v_m \in \Delta_{\bowtie}$) such that $k(N/\bowtie,\mu)>k(N,\mu)$, then \ref{eq:H2} fails at $\tilde{u}$. In this case, we should study the bifurcation problem of $f$ on $N/\bowtie$, or look for zeros of $\tilde{h}$ which do not correspond to a point in $\Delta_{\bowtie}$.

For submaximal colorings, we use condition \ref{eq:H1li}. We say that condition \ref{eq:H1li} holds for $f$ and $N$ if
\begin{equation}
\label{eq:H1li}
\tag{H1a}
\forall_{\tilde{u},\bowtie}\ \   \tilde{h}(\tilde{u})=0 \wedge \bowtie_{=}\prec \bowtie \wedge (\tilde{u}_1 v_1 + \dots +\tilde{u}_m v_m) \notin \Delta_{\bowtie} \Rightarrow \ref{eq:H2} \textrm{ holds for } \tilde{u}.
\end{equation}

\begin{exe}\label{exe:latbalcoldegdet}
As in \ref{exe:latbalcol}, consider again the network \#51 of \cite{K09}. The eigenvalue $0$ of $A_{51}$ is semisimple and has multiplicity $2$. The balanced coloring $\bowtie_=$ is $0$-submaximal of type $2$ with $0$-simple components $\bowtie_3$ and $\bowtie_4$. Let $f:\mathbb{R}\times\mathbb{R}^2\times\mathbb{R}\rightarrow \mathbb{R}$ be a generic regular function with a bifurcation condition associated to $0$, i.e., $f_0=0$. We have that $\dim(\ker(J_f^{N_{51}}))=2$. We choose a basis of $\ker(J_f^{N_{51}})$ such that $v_1\in \Delta_{\bowtie_3}$ and $v_2\in \Delta_{\bowtie_4}$. Let $v_1=(0, 0, 1, 0)\in \Delta_{\bowtie_3}$, $v_2=(-1, 1, 0, 0)\in \Delta_{\bowtie_4}$, $v^*_1=(0, 0, 1, -1)$ and $v^*_2=(-1, 1, 0, 0)/2$, where $v_1^*,v_2^*$ is a basis of $\ran( J_f^{N_{51}})^{\bot}$.
Then
$$h(u_1,u_2,\epsilon,\eta)=\left[\begin{matrix}
\displaystyle f_{0\lambda}\eta u_1 + \frac{f_{00}u_1^2}{2} \\
 \displaystyle  f_{0\lambda}\eta u_2
\end{matrix}\right] + \mathcal{O}(\epsilon).$$
Moreover, $h(u_1,u_2,0,\eta)=0$ if and only if  $(u_1,u_2,\eta)= (\pm 1,0,\mp f_{00}/(2f_{0\lambda}))$ or $(u_1,u_2,\eta)=(0,\pm 1,0)$. We have  
$$\left(\frac{\partial h}{\partial (u,\eta)}\right)_{(\pm 1,0,\mp \frac{f_{00}}{(2f_{0\lambda})})}= \left[\begin{matrix}
\displaystyle 0 & \pm f_{0\lambda} \\
 \displaystyle  \mp \frac{f_{00}}{2} & 0
\end{matrix}\right],$$
$$\left(\frac{\partial h}{\partial (u,\eta)}\right)_{(0,\pm 1,0)}= \left[\begin{matrix}
\displaystyle 0 & 0 \\
 \displaystyle  0 & \pm f_{0\lambda} 
\end{matrix}\right].$$

Since \ref{eq:H2} holds at the zeros $(\pm 1,0)$, then by \ref{prop:bifbrfromzero} there is a bifurcation branch with exactly synchrony $\bowtie_3$ . However condition \ref{eq:H2} fails at the zeros $(0,\pm 1)$ and the dimension of the kernel is equal to $1$. We could apply the Lyapunov-Schmidt Reduction and obtain an one-dimensional bifurcation problem, then we should solve it by finding the lowest non-vanishing terms of the reduced equation. Alternatively, we note that the zero $(0, 1)$ corresponds to a point in the synchrony subspace $\Delta_{\bowtie_4}$ and that $k(N/\bowtie_4, 0)=3$. We can also obtain the bifurcation branch of $f$ on $N$ with synchrony $\bowtie_4$ by studying the bifurcation problem of $f$ on $N/\bowtie_4$.
\end{exe}

\section{Synchrony Branching Lemma -- Odd Dimensional Case}\label{subsec:oddsbl}

Now, we prove the analogous version of the odd dimensional version of the Equivariant Branch Lemma for regular networks.
Recall the notation of \ref{sec:sbl}. 

\begin{teo}\label{prop:bifbraoddker}
Let $N$ be a regular network with valency $\vartheta$ and adjacency matrix $A$, $\bowtie$ a balanced coloring of $N$ and $\mu$ a semisimple eigenvalue of $A$. 
Let $f:\mathbb{R}\times\mathbb{R}^\vartheta\times\mathbb{R}\rightarrow \mathbb{R}$ be a regular function with a bifurcation condition associated to $\mu$ such that $\ker(J_f^N)\cap\Delta_{\bowtie}$ has odd dimension. Assume that condition \ref{eq:H1} holds for $f$ and $N/\bowtie$.
Then generically there is a bifurcation branch of $f$ on $N$ with at least synchrony $\bowtie$. 
Moreover, if $\bowtie$ is $\mu$-maximal, then the bifurcation branch has exactly synchrony $\bowtie$.
\end{teo}

\begin{proof}
Let $N$ be a regular network with valency $\vartheta$ and represented by the adjacency matrix $A$, $\bowtie$ a balanced coloring of $N$ and $\mu$ a semisimple eigenvalue of $A$. 
Let $f:\mathbb{R}\times\mathbb{R}^\vartheta\times\mathbb{R}\rightarrow \mathbb{R}$ be a regular function with a bifurcation condition associated to $\mu$ such that $\ker(J_f^N)\cap\Delta_{\bowtie}$ has odd dimension.

Denote by $Q$ the quotient network of $N$ with respect to $\bowtie$ which is represented by the adjacency matrix $A_Q$. Then $\mu$ is a semisimple eigenvalue of $A_Q$ and $m=\dim(\ker(J_f^Q))=\dim(\ker(J_f^N)\cap\Delta_{\bowtie})$ is odd.

We perform the calculation of \ref{sec:sbl} for the network $Q$ and the generic regular function $f$. Following \ref{sec:sbl}, let $\tilde{h}$ be the vector field in $S^{m-1}$ obtained. 
Since $m$ is odd, from the Poincar\'{e}-Hopf theorem \cite{M65}, we know that there exists at least one $\tilde{u}\in S^{m-1}$ such that $\tilde{h}(\tilde{u})=0$. Then $h(\tilde{u},0,\tilde{\eta}(\tilde{u}))=0$, where $\tilde{\eta}(\tilde{u})$ is defined in \ref{sec:sbl}.

Assuming that condition \ref{eq:H1} holds for $f$ and $N/\bowtie$, we have that \ref{eq:H2} holds for $\tilde{u}$. From \ref{prop:bifbrfromzero},  there exists a bifurcation branch of $f$ on $Q$. Last, we can lift this bifurcation branch of $f$ on $Q$ to a bifurcation branch of $f$ on $N$ with at least synchrony $\bowtie$.

If $\bowtie$ is $\mu$-maximal, then for every $\bowtie'$ such that $\bowtie\prec \bowtie'$ there is no bifurcation branch of $f$ on $N/\bowtie'$. Thus the bifurcation branch has exactly synchrony $\bowtie$.
\end{proof}

We establish the existence of a bifurcation branch, using the result above.

\begin{exe}\label{ex:kerdim3}
Let $N$ be the regular network given by the adjacency matrix:
$$A=\left[\begin{matrix}
 24 & 1 & 2 & 3 & 4 \\
 16 & 0 & 5 & 6 & 7 \\
  8 & 6 & 3 & 8 & 9 \\
  0 & 8 & 9 & 6 & 11\\
 20 & 3 & 4 & 5 & 2
\end{matrix}\right].$$
The eigenvalues of $A$ are: the network valency $34$ with multiplicity $1$; $13$ with multiplicity $1$; and $-4$ with multiplicity $3$.

Let $f:\mathbb{R}\times\mathbb{R}^{34}\times\mathbb{R}\rightarrow \mathbb{R}$ be a regular function with a bifurcation condition associated to the eigenvalue $\mu=-4$. Considering the trivial balanced coloring, $\bowtie_=$, we have that $\dim(\ker^*(J_f^N)\cap\Delta_{\bowtie})=3$. Assume that condition \ref{eq:H1} holds for $f$ and $N$. Then there exists a bifurcation branch of $f$ on $N$, by \ref{prop:bifbraoddker}.

The network $N$ has no non-trivial balanced colorings and the unique eigenvalue of the adjacency matrix associated to $N/\bowtie_0$ is $34$. So $\bowtie_{=}$ is $(-4)$-maximal and the bifurcation branch has no synchrony. 
\end{exe}

In the next two examples, we do not explicitly present the networks but we assume that the networks satisfy some conditions. Since the number of cells in a network is not restricted, those conditions may be solved with as many variables as it is needed to consider. The next example shows that we cannot remove the odd dimension condition in \ref{prop:bifbraoddker}.

\begin{exe}\label{exe:ker2deg3nobifbra}
Let $N$ be a network with adjacency matrix $A$, $\mu$ a semisimple eigenvalue of $A$ with multiplicity $2$. Let $f$ be a generic regular function with a bifurcation condition associated to $\mu$,  $(v_1,v_2)$ be a basis for $\ker(J_f^N)$ and $(v^*_1,v^*_2)$ be a basis for $\ran( J_f^N)^{\bot}$. We will assume that $A$, $(v_1,v_2)$ and $(v^*_1,v^*_2)$ respect the following conditions:
$$\langle v^*_1,v_1 \rangle=\langle v^*_2,v_2 \rangle=1, \quad \quad \langle v^*_2,v_1 \rangle=\langle v^*_1,v_2 \rangle=0,$$
$$\langle v^*_i,v_1\ast v_1 \rangle=\langle v^*_i,v_1\ast v_2 \rangle=\langle v^*_i,v_2\ast v_2 \rangle=0,\quad i=1,2,$$
$$\langle v^*_1, v_1\ast v_1\ast v_1)\rangle =\langle v^*_2, v_1\ast v_1\ast v_2)\rangle,\quad \langle v^*_1, v_2\ast v_2\ast v_1)\rangle=\langle v^*_2, v_2\ast v_2\ast v_2)\rangle,$$
$$\langle v^*_1,  3 (v_1\ast (A(v_1\ast v_1)))\rangle =\langle v^*_2,  2 v_1\ast (A(v_1\ast v_2))+ v_2\ast (A(v_1\ast v_1))\rangle,$$
$$\langle v^*_2,  3 (v_2\ast (A(v_2\ast v_2)))\rangle=\langle v^*_1,  2 v_2\ast (A(v_2\ast v_1))+ v_1\ast (A(v_2\ast v_2))\rangle,$$
$$0\neq \langle v^*_1, v_1\ast v_1\ast v_2)\rangle =\langle v^*_1, v_2\ast v_2\ast v_2)\rangle=-\langle v^*_2, v_1\ast v_2\ast v_2)\rangle =-\langle v^*_2, v_1\ast v_1\ast v_1)\rangle,$$
$$\begin{aligned} 
0 \neq \langle v^*_1,  2 v_1\ast (A(v_1\ast v_2))+ v_2\ast (A(v_1\ast v_1))\rangle&=\langle v^*_1,  3 (v_2\ast (A(v_2\ast v_2)))\rangle=\\=-\langle v^*_2,  2 v_2\ast (A(v_1\ast v_2))+ v_1\ast (A(v_2\ast v_2))\rangle&=-\langle v^*_2,  3 (v_1\ast (A(v_1\ast v_1)))\rangle.\end{aligned}$$
Then $k(N,\mu)=3$. Let $g=(g_1,g_2):\mathbb{R}^2\times \mathbb{R}\rightarrow \mathbb{R}^2$ be the reduced equation obtained by the Lyapunov-Schmidt Reduction. The previous equalities imply for every regular function $f$ that
$$\frac{\partial^3 g_1}{\partial y_1^3}=\frac{\partial^3 g_2}{\partial y_1^2 \partial y_2},\quad \quad \frac{\partial^3 g_1}{\partial y_1 \partial y_2^2}=\frac{\partial^3 g_2}{\partial y_2^3},$$
$$\alpha=\frac{\partial^3 g_1}{\partial y_1^2\partial y_2}=\frac{\partial^3 g_1}{\partial y_2^3}=-\frac{\partial^3 g_2}{\partial y_1^3}=-\frac{\partial^3 g_2}{\partial y_1 \partial y_2^2}\neq 0$$
at $(y_1,y_2,\lambda)=(0,0,0)$. Then the vector field, $\tilde{h}$, on the $1$-sphere is given by
$$\tilde{h}(u_1,u_2)=\left[\begin{matrix}
\alpha u_2(u_1^2+u_2^2)^2\\
- \alpha u_1(u_1^2+u_2^2)^2
\end{matrix}\right]= \alpha \left[\begin{matrix}
u_2\\
- u_1
\end{matrix}\right],$$
where $(u_1,u_2)\in S^1$. We have that $\tilde{h}(u_1,u_2)\neq 0$ for every $(u_1,u_2)\in S^1$. By \ref{prop:bifbrtozero}, there is no bifurcation branch of $f$ on $N$.
\end{exe}

In the next example, we lift the network of the previous example to a network with one more cell. This example shows that there are zeros of the vector field on the sphere which do not correspond to a bifurcation branch and the relevance of \ref{eq:H2} in \ref{prop:bifbrfromzero}.

\begin{exe}
Let $N$, $A$, $\mu$, $v_1$, $v_2$, $v^*_1$ and $v^*_2$ as in \ref{exe:ker2deg3nobifbra}. Assume that $\mu=0$. Fix a cell $c$ of $N$. Let $\hat{N}$ be the network with $|N|+1$ cells given by the following adjacency matrix:
$$\hat{A}=\left[\begin{matrix}
A & 0\\ 
R & 0 
\end{matrix}\right],$$
where $R=(A_{c\:d})_{d\in N}$ is a $1\times |N|$-matrix. Denote by $c'$ the new cell of $\hat{N}$ and by $\bowtie$ the balanced coloring of $\hat{N}$ given by $c\bowtie c'$. The network $N$ is a quotient network of $\hat{N}$ with respect to $\bowtie$.
Note that $0$ is a semisimple eigenvalue of $\hat{A}$ with multiplicity $3$.

Let $f$ be a generic regular function with a bifurcation condition associated to $\mu=0$, $\hat{v}_1=(v_1,(v_1)_c)$, $\hat{v}_2=(v_2,(v_2)_c)$, $\hat{v}_3=(0,\dots,0,1)$, $\hat{v}^*_1=(v_1^*,0)$, $\hat{v}^*_2=(v_2^*,0)$ and $\hat{v}^*_3\in\mathbb{R}^{|\hat{N}|}$ such that $(\hat{v}^*_3)_c=-1$, $(\hat{v}^*_3)_{c'}=1$ and $(\hat{v}^*_3)_{d}=0$ otherwise. Then $(\hat{v}_1,\hat{v}_2, \hat{v}_3)$ is a basis of $\ker(J_f^{\hat{N}})$ and $(\hat{v}^*_1,\hat{v}^*_2, \hat{v}^*_3)$ is a basis of $\ran( J_f^{\hat{N}})^{\bot}$. Let $\hat{g}=(\hat{g}_1,\hat{g}_2,\hat{g}_3):\mathbb{R}^3\times \mathbb{R}\rightarrow \mathbb{R}^3$ be the reduced equation obtained by the Lyapunov-Schmidt Reduction of $f^{\hat{N}}$. Note that 
$$\frac{\partial^2 g_3}{\partial y_1\partial y_3}=f_{00}(v_1)_c,\quad \quad\frac{\partial^2 g_3}{\partial y_2\partial y_3}=f_{00}(v_2)_c,\quad \quad\frac{\partial^2 g_3}{\partial y_3^2}=f_{00},\quad \quad \frac{\partial^2 g_i}{\partial y_j \partial y_l}=0,$$
for any other $i$, $j$ and $l$. Then $k(\hat{N},0)=2$ and 
$$h(u_1,u_2,u_3,0,\eta)=\left[\begin{matrix}
f_{0\lambda}\eta u_1\\
f_{0\lambda}\eta u_2\\
\displaystyle f_{0\lambda}\eta u_3+\frac{f_{00}}{2}u_3^2 + f_{00}(v_1)_c u_1 u_3 + f_{00}(v_2)_c u_2 u_3
\end{matrix}\right],$$
where $(u_1,u_2,u_3)\in S^2$. 
We have that $h(0,0, 1,- f_{00}/(2f_{0\lambda}))= 0$, condition~\ref{eq:H2} holds for $(0,0, 1)$ and it leads to a bifurcation branch of $f$ on $\hat{N}$. 

On the other hand $h(u_1,u_2,0,0)=0$, condition~\ref{eq:H2} does not hold for $(u_1,u_2,0)$, where $(u_1,u_2,0,0)\in S^{2}\times\mathbb{R}$, and it does not lead to a bifurcation branch of $f$ on $\hat{N}$. Otherwise this bifurcation branch would be inside $\Delta_{\bowtie}\subseteq \mathbb{R}^{|\hat{N}|}$, since $\hat{v}_1,\hat{v}_2\in \Delta_{\bowtie}$, and would lead to a bifurcation branch of $f$ on $N$. But, as we saw, there is no bifurcation branch of $f$ on $N$.
\end{exe}

\section{Synchrony Branching Lemma -- Two dimensional case}\label{subsec:twosbl}

In this section, we study the smallest case not included in the results of the previous section, i.e., when the semisimple eigenvalue $\mu$ has multiplicity $m=2$ and $k(N,\mu)$ is even. 
We give conditions for the existence of bifurcation branches with maximal or submaximal synchrony. 

Let $N$ be a regular network with valency $\vartheta$ and represented by the adjacency matrix $A$, $\mu$ a semisimple eigenvalue of $A$ and $f:\mathbb{R}\times\mathbb{R}^\vartheta\times\mathbb{R}\rightarrow \mathbb{R}$ a generic regular function with a bifurcation condition associated to $\mu$ such that $m=\dim(\ker(J_f^N))=2$.

Taking into account the calculations in \ref{sec:sbl}, when $m=\dim(\ker(J_f^N)=2$, we have the following vector field on the $1$-sphere
$$\tilde{h}(u_1,u_2)= Q_k(u_1,u_2)-\langle Q_k(u_1,u_2),(u_1,u_2)\rangle (u_1,u_2),$$
where $u_1^2+u_2^2=1$, $k=k(N,\mu)$. Transforming the variables $(u_1,u_2)$ to the angular variable $\theta$, where $(u_1,u_2)=(\cos (\theta),\sin(\theta))$. The dynamical system $\dot{u}=\tilde{h}(u)$ is equivalent to the dynamical system $\dot{\theta}=\Theta(\theta)$ given by
$$\Theta(\theta)= \cos(\theta) q_2(\cos(\theta),\sin(\theta))-\sin(\theta) q_1(\cos(\theta),\sin(\theta)),$$
where $Q_k(u_1,u_2)=(q_1(u_1,u_2),q_2(u_1,u_2))$. So, finding zeros of $\tilde{h}$ is equivalent to solve the equation
$$\Theta(\theta)=0.$$

\begin{teo}\label{prop:bifbranchdim2max}
Let $N$ be a regular network with valency $\vartheta$ and adjacency matrix $A$, $\bowtie$ a balanced coloring of $N$ and $\mu$ a semisimple eigenvalue of $A$. Let $f:\mathbb{R}\times\mathbb{R}^\vartheta\times\mathbb{R}\rightarrow \mathbb{R}$ be a regular function with a bifurcation condition associated to $\mu$ and such that $\dim(\ker(J_f^N)\cap\Delta_{\bowtie})=2$. Assume that condition \ref{eq:H1} holds for $f$ and $N/\bowtie$. If $\bowtie$ $\mu$-maximal and $k(N/\bowtie,\mu)$ is even, then generically there is a bifurcation branch of $f$ on $N$ with exactly synchrony $\bowtie$.
\end{teo}

\begin{proof}
Let $N$ be a regular network with valency $\vartheta$, $\bowtie$ a balanced coloring of $N$ and $\mu$ a semisimple eigenvalue of the adjacency matrix of $N$. 
Take $f:\mathbb{R}\times\mathbb{R}^\vartheta\times\mathbb{R}\rightarrow \mathbb{R}$ to be a regular function with a bifurcation condition associated to $\mu$ such that $\dim(\ker(J_f^N)\cap\Delta_{\bowtie})=2$.

Let $A_{\bowtie}$ to be the adjacency matrix of $N/\bowtie$. Then $m=\dim(\ker(J_f^{N/\bowtie}))=\dim(\ker(J_f^N)\cap\Delta_{\bowtie})=2$ and $\mu$ is a semisimple eigenvalue of $A_{\bowtie}$.

Performing the calculations of \ref{sec:sbl} and of this section for the network $N/\bowtie$ and the regular function $f$, we consider the function $\Theta(\theta)$ given by
$$\Theta(\theta)= \cos(\theta) q_2(\cos(\theta),\sin(\theta))-\sin(\theta) q_1(\cos(\theta),\sin(\theta)),$$
where $Q_k(u_1,u_2)=(q_1(u_1,u_2),q_2(u_1,u_2))$ and $k=k(N/\bowtie,\mu)$. We look for solutions of $\Theta(\theta)=0$.

Suppose that $\bowtie$ is $\mu$-maximal and $k(N/\bowtie,\mu)$ is even. Then $Q_k(u_1,u_2)=Q_k(-u_1,-u_2)$ and
$$\Theta(\theta+\pi)=-\Theta(\theta).$$

By the intermediate value theorem, we know that it must exists $\tilde{\theta}$ such that $\Theta(\tilde{\theta})=0$. Consider $\tilde{u}=(\cos(\tilde{\theta}),\sin(\tilde{\theta}))$, then $\tilde{h}(\tilde{u})=0$. 
Assuming that condition~\ref{eq:H1} holds for $f$ and $N/\bowtie$, we know that \ref{eq:H2} holds for $\tilde{u}$. From \ref{prop:bifbrfromzero}, there exists a bifurcation branch of $f$ on $N/\bowtie$. Lifting this bifurcation branch, we obtain a bifurcation branch of $f$ on $N$ with at least synchrony $\bowtie$. Since $\bowtie$ is $\mu$-maximal this bifurcation branch of $f$ on $N$ has exactly synchrony $\bowtie$. 
\end{proof}

Returning to the beginning of this section. If $k(N,\mu)=2$, then 
\begin{equation}\label{eq:blowupdim2}
q_i(u_1,u_2)= \beta\left(\frac{d_{i11}}{2}u_1^2+ d_{i12}u_1u_2+\frac{d_{i22}}{2}u_2^2\right),
\end{equation}
where $i=1,2$, $\beta=(f_{00}+2\mu f_{01}+\mu f_{11}-\mu f_{1\vartheta}+\mu^2 f_{1\vartheta})$ and for $l_1,l_2=1,2$ $$d_{il_1 l_2}=\langle v^*_i,v_{l_1}\ast v_{l_2}\rangle.$$

\begin{teo}\label{prop:bifbranchdim2submax}
Let $N$ be a regular network with valency $\vartheta$ and adjacency matrix $A$, $\bowtie$ a balanced coloring of $N$ and $\mu$ a semisimple eigenvalue of $A$. Let $f:\mathbb{R}\times\mathbb{R}^\vartheta\times\mathbb{R}\rightarrow \mathbb{R}$ be a regular function with a bifurcation condition associated to $\mu$ and such that $\dim(\ker(J_f^N)\cap\Delta_{\bowtie})=2$. Assume that condition \ref{eq:H1li} holds for $f$ and $N/\bowtie$ and $k(N/\bowtie,\mu)=2$. 

\noindent
$(i)$ If $\bowtie$ is $\mu$-submaximal of type $1$. Then generically there is a bifurcation branch of $f$ on $N$ with  exactly synchrony $\bowtie$ if and only if 
\begin{equation}\label{eq:condsubmaxtype1}
(d_{222} -2d_{112})^2 \geq 4d_{122}(d_{111}-2d_{212}).
\end{equation}

\noindent
$(ii)$ If $\bowtie$ is $\mu$-submaximal of type $2$. Then generically there is a bifurcation branch of $f$ on $N$ with  exactly synchrony $\bowtie$ if and only if 
\begin{equation}\label{eq:condsubmaxtype2}
(2d_{212}- d_{111})(2d_{112}- d_{222})\neq 0.
\end{equation}

\end{teo}

\begin{proof}
Let $N$ be a regular network with valency $\vartheta$ and adjacency matrix $A$, $\bowtie$ a balanced coloring of $N$ and $\mu$ a semisimple eigenvalue of $A$. Let $f:\mathbb{R}\times\mathbb{R}^\vartheta\times\mathbb{R}\rightarrow \mathbb{R}$ be a regular function with a bifurcation condition associated to $\mu$ and such that $\dim(\ker(J_f^N)\cap\Delta_{\bowtie})=2$ and $k(N/\bowtie,\mu)=2$.

Let $A_{\bowtie}$ to be the adjacency matrix of $N/\bowtie$. Then $m=\dim(\ker(J_f^{N/\bowtie}))=\dim(\ker(J_f^N)\cap\Delta_{\bowtie})=2$ and $\mu$ is a semisimple eigenvalue of $A_{\bowtie}$.

Performing the calculations of \ref{sec:sbl} and of this section for the network $N/\bowtie$ and the regular function $f$, we consider the function $\Theta(\theta)$ given by
$$\Theta(\theta)= \cos(\theta) q_2(\cos(\theta),\sin(\theta))-\sin(\theta) q_1(\cos(\theta),\sin(\theta)),$$
where $Q_k(u_1,u_2)=(q_1(u_1,u_2),q_2(u_1,u_2))$ and $k=k(N/\bowtie,\mu)$. We look for solutions of $\Theta(\theta)=0$.

\noindent
$(i)$ Suppose that $\bowtie$ is $\mu$-submaximal of type $1$ with a $\mu$-simple component $\bowtie_1$. We denote by $\bowtie'_1$ the balanced coloring of $N/\bowtie$ that corresponds to the balanced coloring $\bowtie_1$ of $N$. Now, we return to the calculations of \ref{sec:sbl} for the network $N/\bowtie$ and we choose a basis $v_1, v_2$ of $\ker(J_f^{N/\bowtie})$ such that $v_1\in \Delta_{\bowtie'_1}$. 

Note that $d_{211}=0$, for the following reason. From \ref{prop:bifbraoddker} and $\dim(\ker(J_f^{N/\bowtie})\cap\Delta_{\bowtie'_1})=1$, there exists a bifurcation branch of $f$ on $N/\bowtie$ with synchrony $\bowtie'_1$. By \ref{prop:bifbrtozero}, $\tilde{h}(\pm 1,0)=0$. So $d_{211}=0$. (This can be also shown using the fact that $\frac{\partial^2 g_2}{\partial y_1\partial y_1}(0,0)=0$, since $\Delta_{\bowtie'_1}$ is invariant.) 

Using the expansion of $q_1$ and $q_2$ presented in \ref{eq:blowupdim2}, we have that 
$$\Theta(\theta)= \beta\left(u_1\left( d_{212}u_1u_2+\frac{d_{222}u_2^2}{2}\right)-u_2 \left(\frac{d_{111}u_1^2}{2}+ d_{112}u_1u_2+\frac{d_{122}u_2^2}{2}\right)\right),$$
where $u_1=\cos(\theta)$, $u_2=\sin(\theta)$ and $\beta=(f_{00}+2\mu f_{01}+\mu f_{11}-\mu f_{1\vartheta}+\mu^2 f_{1\vartheta})\neq0$ generically. We have that $\Theta(\theta)=0$, if $\sin(\theta)=0$, however those zeros correspond to the known bifurcation branch of $f$ on $N/\bowtie$ with synchrony $\bowtie'_1$. For $\sin(\theta)\neq 0 $, we have that 
$$\Theta(\theta)=0 \Leftrightarrow (2d_{212}-d_{111})\left(\frac{\cos(\theta)}{\sin(\theta)}\right)^2 + (d_{222} -2d_{112})\frac{\cos(\theta)}{\sin(\theta)} - d_{122} = 0.$$
Define $x=\cos(\theta)/\sin(\theta)$ and consider the equation 
\begin{equation}\label{eq:eqsubmaxtype1}
(2d_{212}-d_{111})x^2 + (d_{222} -2d_{112})x - d_{122} = 0
\end{equation} that has a real solution if and only if \ref{eq:condsubmaxtype1} holds.

If \ref{eq:condsubmaxtype1} holds, then there exists a solution $\tilde{x}$ of \ref{eq:eqsubmaxtype1}. Since the image of $\theta\mapsto \cos(\theta)/\sin(\theta)$, $\sin(\theta)\neq 0$ is the entire real line. There exists $\tilde{\theta}$ such that $\tilde{x}=\cos(\tilde{\theta})/\sin(\tilde{\theta})$ and $\Theta(\tilde{\theta})=0$. Consider $\tilde{u}=(\cos(\tilde{\theta}),\sin(\tilde{\theta}))$, then $\tilde{h}(\tilde{u})=0$. 
Assuming that \ref{eq:H1li} holds for $f$ and $N/\bowtie$, we know that \ref{eq:H2} holds for $\tilde{u}$. From \ref{prop:bifbrfromzero}, there exists a bifurcation branch of $f$ on $N/\bowtie$. This bifurcation branch has no synchrony in $N/\bowtie$, since $\bowtie$ is $\mu$-submaximal and $ \sin(\tilde{\theta})\neq 0$. Lifting this bifurcation branch to $N$, we obtain a bifurcation branch of $f$ on $N$ with exactly synchrony $\bowtie$.

Suppose by contradiction that there exists a bifurcation branch of $f$ on $N$ with exactly synchrony $\bowtie$ and \ref{eq:condsubmaxtype1} does not hold. We have by \ref{prop:bifbrtozero} that there exists $(u_1,u_2)\in S^{1}$ such that $u_2\neq 0$ and $\tilde{h}(u_1,u_2)=0$. So $x=u_1/u_2$ is a real solution of equation \ref{eq:eqsubmaxtype1}, which  is an absurd since we are supposing that \ref{eq:condsubmaxtype1} does not hold. This proves $(i)$.

\noindent
$(ii)$ Suppose that $\bowtie$ is $\mu$-submaximal of type $2$ with  $\mu$-simple components $\bowtie_1$ and $\bowtie_2$. We denote by $\bowtie'_1, \bowtie'_2$ the balanced colorings of $N/\bowtie$ that corresponds to the balanced colorings $\bowtie_1,\bowtie_2$ of $N$, respectively. In the calculations of \ref{sec:sbl} for the network $N/\bowtie$, we choose a basis $v_1, v_2$ of $\ker(J_f^{N/\bowtie})$ such that $v_1\in \Delta_{\bowtie'_1}$ and $v_2\in \Delta_{\bowtie'_2}$. 
 
As before, we note that $d_{211}=0$ and $d_{122}=0$. So
$$\Theta(\theta)= \beta\left(u_1u_2\left( d_{212}u_1+\frac{d_{222}u_2}{2}\right)-u_2u_1 \left(\frac{d_{111}u_1}{2}+ d_{112}u_2\right)\right),$$
where $u_1=\cos(\theta)$ and $u_2=\sin(\theta)$. We have that $\Theta(\theta)=0$, if $\sin(\theta)=0$  or $\cos(\theta)=0$, however those zeros correspond to the known bifurcation branch of $f$ on $N/\bowtie$ with synchrony $\bowtie'_1$ or $\bowtie'_2$. For $\cos(\theta),\sin(\theta)\neq 0$, we have that 
$$\Theta(\theta)=0 \Leftrightarrow  (2d_{212}- d_{111})\cos(\theta)=  (2d_{112}- d_{222})\sin(\theta).$$
has a solution such that $\cos(\theta),\sin(\theta)\neq 0$ if and only if  \ref{eq:condsubmaxtype2} holds.

If \ref{eq:condsubmaxtype2} holds, let $\tilde{\theta}$ be a solution of the equation above, i.e., $\Theta(\tilde{\theta})=0$. Assuming that condition~\ref{eq:H1li} holds for $f$ and $N/\bowtie$, we know that \ref{eq:H2} holds for  $\tilde{u}=(\cos(\tilde{\theta}),\sin(\tilde{\theta}))$. 
From \ref{prop:bifbrfromzero}, there exists a bifurcation branch of $f$ on $N/\bowtie$
which has no synchrony in $N/\bowtie$, since $\bowtie$ is $\mu$-submaximal and $\cos(\theta),\sin(\theta)\neq 0$. Lifting this bifurcation branch to $N$, we obtain a bifurcation branch of $f$ on $N$ with exactly synchrony $\bowtie$.

Suppose by contradiction that there exists a bifurcation branch of $f$ on $N$ with exactly synchrony $\bowtie$ and \ref{eq:condsubmaxtype2} does not hold. We have by \ref{prop:bifbrtozero} that there exists $(u_1,u_2)\in S^{1}$ such that $u_1,u_2\neq 0$ and $\tilde{h}(u_1,u_2)=0$. So $(2d_{212}- d_{111})u_1=  (2d_{112}- d_{222})u_2$ and $u_1,u_2\neq 0$. This is an absurd since $(2d_{212}- d_{111})(2d_{112}- d_{222})= 0$. This proves $(ii)$.
\end{proof}

Now, we study some examples where the trivial balanced coloring $\bowtie_{=}$ is maximal or submaximal. We see, in particular, that there are networks with similar lattice structures but different synchrony-breaking bifurcations.

\begin{exe}
Consider the network $N_1$ given by the adjacency matrix
$$A_1=\left[\begin{matrix}
343 & 430 & 86 & 129\\ 
377 & 453 & 77 & 81\\ 
47 & 214 & 166 & 561\\ 
432 & 494 & 62 & 0
\end{matrix}\right].$$
The eigenvalues of $A_1$ are the network valency $988$, $-24$ and $-1$ with multiplicity $1$, $1$ and $2$, respectively. They are semisimple. We consider the trivial balanced coloring $\bowtie_{=}$ and the bifurcations associated to the eigenvalue $\mu=-1$ of $A_1$. The network $N_1$ does not have any non-trivial balanced coloring. So the balanced coloring $\bowtie_=$ is maximal and is $(-1)$-maximal. 

Let $f:\mathbb{R}\times\mathbb{R}^{988}\times\mathbb{R}\rightarrow \mathbb{R}$ be a generic regular function with a bifurcation condition associated to $-1$, i.e., $f_0-f_1=0$. We have that $\dim(\ker(J_f^{N_1}))=2$. Let $v_1=(1,-1,1,0)$ and $v_2=(8,-7,0,2)$ be a basis of $\ker(J_f^{N_1})$. Let $(v^*_1,v^*_2)$ be a basis of $\ran( J_f^{N_1})^{\bot}$ such that $\langle v^*_i, v_j \rangle$ is $1$ if $i=j$ and $0$ otherwise, for $i,j=1,2$. Then $d_{abc}\neq 0$ for every $a,b,c=1,2$ and $k(N_1,\mu)=2$ is even.

By \ref{prop:bifbranchdim2max}, there is a bifurcation branch of $f$ on $N_1$ without synchrony, if condition~\ref{eq:H1} holds for $f$ and $N_1$. 
\end{exe}

\begin{exe}
Consider the network $N_2$ given by the adjacency matrix
$$A_2=\left[\begin{matrix}
2 & 0 & 0\\ 
0 & 2 & 0\\
1 & 1 & 0
\end{matrix}\right].$$
The eigenvalues of $A_2$ are the network valency $2$ and $0$ with multiplicity $2$ and $1$, respectively. They are semisimple. We  consider the trivial balanced coloring $\bowtie_{=}$ and bifurcations associated to the eigenvalue $2$ of $A_2$. The network $N_2$ has only one non-trivial balanced coloring: $\bowtie_1$ given by the classes: $\{\{1,2\},\{3\}\}$.
The balanced coloring $\bowtie_=$ is $2$-submaximal of type $1$, with $2$-simple component $\bowtie_1$. See \ref{fig:latbalcolorsubmaxtype1}.
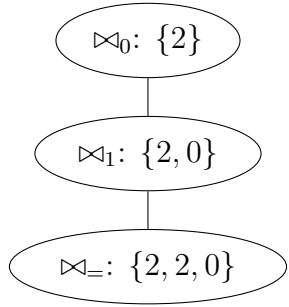
\begin{figure}[ht]
\center
\begin{tikzpicture}[node distance=0.5cm and 0.7cm]
\node (n1) [ellipse,draw]  {$\bowtie_{0}$: $\{2\}$};
\node (n2) [ellipse,draw]  [below=of n1] {$\bowtie_1$: $\{2,0\}$};
\node (n3) [ellipse,draw]  [below=of n2] {$\bowtie_{=}$: $\{2,2,0\}$};

\draw (n1) to (n2);
\draw (n2) to  (n3);
\end{tikzpicture}
\caption{Balanced colorings of $N_2$ and the eigenvalues of the adjacency matrices associated to the corresponding quotient networks.}
\label{fig:latbalcolorsubmaxtype1}
\end{figure}

Let $f:\mathbb{R}\times\mathbb{R}^2\times\mathbb{R}\rightarrow \mathbb{R}$ be a generic regular function with a bifurcation condition associated to $2$, i.e., $f_0+2f_1=0$. We have that $\dim(\ker(J_f^{N_2}))=2$. We choose a basis of $\ker(J_f^{N_2})$ such that $v_1\in \Delta_1$. Let $v_1=(1,1,1)\in \Delta_1$ and $v_2=(1,-1,0)$. Let $v^*_1=(1/2,1/2,0)$ and $v^*_2=(1/2,-1/2,0)$ that form a basis of $\ran( J_f^{N_2})^{\bot}$. Then
$$d_{112}=d_{211}=d_{222}=0,\quad \quad d_{111}=d_{122}=d_{212}=1.$$
Therefore $k(N_2,\mu)=2$ and \ref{eq:condsubmaxtype1} holds.

By explicit calculations, we can see that condition \ref{eq:H1li} holds whenever $f_{0\lambda}+2f_{1\lambda}\neq 0$ and $f_{00}+4f_{01}+2f_{11}+2f_{19}\neq 0$.
Then there is a bifurcation branch of $f$ on $N_2$ without synchrony, by \ref{prop:bifbranchdim2submax}~$(i)$. 
\end{exe}

\begin{exe}\label{exe:latbalcoldegdet2}
Consider the network \#29 of \cite{K09} (\ref{fig:net29}) which will be denoted by $N_{29}$ and has the adjacency matrix
$$A_{29}=\left[\begin{matrix}
0 & 0 & 0 & 2\\ 
0 & 0 & 0 & 2\\
0 & 1 & 1 & 0\\
0 & 1 & 1 & 0
\end{matrix}\right].$$
The eigenvalues of $A_{29}$ are the network valency $2$, $-1$ and $0$ with multiplicity $1$, $1$ and $2$, respectively, which are semisimple.
The network $N_{29}$ has four non-trivial balanced colorings $\bowtie_1=\{\{1,2\},\{3,4\}\}$, $\bowtie_2=\{\{1\},\{2,3,4\}\}$, $\bowtie_3=\{\{1,2\},\{3\},\{4\}\}$ and $\bowtie_4=\{\{1\},\{2\},\{3,4\}\}$.
The balanced coloring $\bowtie_=$ is $0$-submaximal of type $2$ with $0$-simple components $\bowtie_3$ and $\bowtie_4$. See \ref{fig:latbalcoloreige29}. 

\begin{figure}[ht]
\center
\begin{tikzpicture}[node distance=0.5cm and 0.7cm]
\node (n1) [ellipse,draw]  {$\bowtie_{0}$: $\{2\}$};
\node (n2) [ellipse,draw]  [below left=of n1] {$\bowtie_2$: $\{2,0\}$};
\node (n3) [ellipse,draw]  [below right=of n1] {$\bowtie_1$: $\{2,-1\}$};
\node (n4) [ellipse,draw]  [below=of n2] {$\bowtie_4$: $\{2,-1,0\}$};
\node (n5) [ellipse,draw]  [below=of n3] {$\bowtie_3$: $\{2,-1,0\}$};
\node (n6) [ellipse,draw]  [below right=of n4, xshift=-1cm] {$\bowtie_{=}$: $\{2,-1,0,0\}$};

\draw (n1) to  (n2);
\draw (n1) to  (n3);
\draw (n2) to  (n4);
\draw (n3) to  (n4);
\draw (n3) to  (n5);
\draw (n4) to  (n6);
\draw (n5) to  (n6);
\end{tikzpicture}
\caption{Balanced colorings of network $N_{29}$ and the eigenvalues of the adjacency matrices associated to the corresponding quotient networks.}
\label{fig:latbalcoloreige29}
\end{figure}
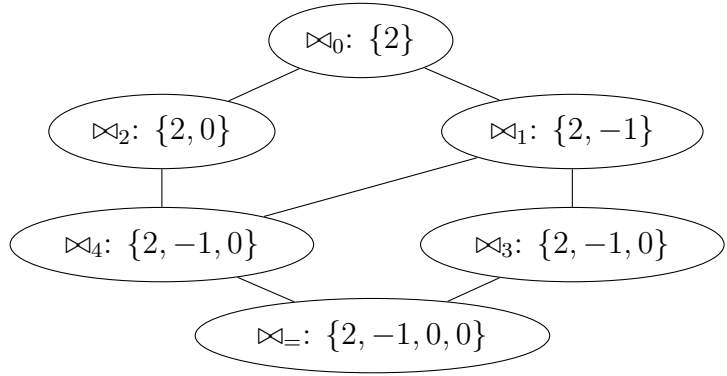

 We consider the trivial balanced coloring $\bowtie_{=}$. 
Let $f:\mathbb{R}\times\mathbb{R}^4\times\mathbb{R}\rightarrow \mathbb{R}$ be a generic regular function with a bifurcation condition associated to $0$, i.e., $f_0=0$. We have that $\dim(\ker(J_f^{N_{29}}))=2$. We choose a basis of $\ker(J_f^{N_{29}})$ such that $v_1\in \Delta_{\bowtie_3}$ and $v_2\in \Delta_{\bowtie_4}$. Let $v_1=(1, 1, -1, 0)\in \Delta_{\bowtie_3}$ and $v_2=(1, 0, 0, 0)\in \Delta_{\bowtie_4}$. Let $v^*_1=(0, 0, -1, 1)$ and $v^*_2=(1, -1, 0, 0)$ that form a basis of $\ran( J_f^{N_{29}})^{\bot}$. Then
$$d_{111}=-1,\quad d_{112}=d_{122}=d_{211}=0,\quad d_{212}=1,\quad d_{222}=1.$$
Therefore $k(N_{29},\mu)=2$ and \ref{eq:condsubmaxtype2} holds. 

By explicit calculations, we can see that condition \ref{eq:H1li} holds, whenever $f_{0\lambda}\neq 0$ and $f_{00}\neq 0$.
Then there is a bifurcation branch of $f$ on $N_{29}$ without synchrony, by \ref{prop:bifbranchdim2submax}~$(ii)$. 
\end{exe}

\begin{obs}
Note that the networks $N_{51}$ and $N_{29}$ share the same lattice structure. However, the network $N_{51}$ does not support a bifurcation branch without synchrony, \ref{exe:latbalcoldegdet}, and the network $N_{29}$ supports a bifurcation branch without synchrony, \ref{exe:latbalcoldegdet2}.
\end{obs}

\section*{Acknowledgments}

I thank Manuela Aguiar and Ana Paula Dias  for helpful discussions and guidance throughout this work.

\end{document}